\documentclass[12pt]{amsart}

\usepackage{amsmath}
\usepackage{amscd}
\usepackage{amssymb}
\usepackage{latexsym}
\usepackage{stmaryrd}
\usepackage[backend=bibtex]{biblatex}

\usepackage{mathrsfs}

\newtheorem{theorem}{Theorem}[section]
\newtheorem*{theorem*}{Theorem}

\newtheorem{lemma}[theorem]{Lemma}
\newtheorem*{lemma*}{Lemma}
\newtheorem{corollary}[theorem]{Corollary}
\newtheorem{proposition}[theorem]{Proposition}

\newtheorem{remark}[theorem]{Remark}
\newtheorem{definition}[theorem]{Definition}

\newcommand{\am}{\alpha^{(m)}}
\newcommand{\J}{\mathcal{J}}
\newcommand{\I}{\mathcal{I}}
\newcommand{\HH}{\mathcal{H}}
\newcommand{\hh}{\mathcal{H}}

\newcommand{\A}{\mathcal{A}}
\newcommand{\B}{\mathcal{B}}
\newcommand{\M}{\mathcal{M}}

\newcommand{\Q}{\mathcal{Q}}

\newcommand{\C}{\mathbb{C}}

\newcommand{\ATP}{\mathcal{A}^+_{tail}}
\newcommand{\ATN}{\mathcal{A}^-_{tail}}
\newcommand{\p}{{\bf P}}
\newcommand{\um}{ u^{(m)}}
\newcommand{\ub}{ u^{(b)}}
\usepackage[all]{xy}

\bibliography{references}
\begin{document}
\title[noncommutative spreadability]{ Extended de Finetti  theorems for boolean independence and monotone independence}
\author{Weihua Liu}
\maketitle
\begin{abstract}
We construct several new spaces of quantum sequences  and their quantum families of maps in sense of So{\l}tan.   Then, we introduce noncommutative distributional symmetries  associated with these quantum maps and study simple relations between them.    We will focus on studying two kinds of noncommutative distributional symmetries: monotone spreadability and boolean spreadability.  We provide an example of a spreadable sequence of random variables for which the usual unilateral shift is an unbounded map. As a result, it is natural to study bilateral sequences of random objects, which are indexed by integers, rather than unilateral sequences. In the end of the paper, we will show  Ryll-Nardzewski type theorems for monotone independence and boolean independence:  Roughly speaking, an infinite bilateral sequence of random variables is monotonically(boolean) spreadable if  and only if the variables are identically distributed and monotone(boolean) with respect to the conditional expectation onto its tail algebra. For an infinite sequence of noncommutative random variables, boolean spreadability is equivalent to boolean exchangeability.
\end{abstract}

\section{Introduction}
The characterization of random objects with distributional symmetries is an important object in modern probability and the recent context of Kallenberg \cite{Ka} provides a comprehensive treatment of distributional symmetries in classical probability.   A finite sequence of random variables $(\xi_1,\xi_2,...,\xi_n)$  is said to be  exchangeable if   

$$ (\xi_1,...,\xi_n)\overset{d}{=}(\xi_{\sigma(1)},...,\xi_{\sigma(n)}),\,\,\,\, \forall \sigma\in S_n,$$
where $S_n$ is the permutation group of $n$ elements and $\overset{d}{=}$ meas the joint distribution of the two sequences are the same.
Compare with  exchangeability, there is a weaker condition  of spreadability: $(\xi_1,...,\xi_n)$ is said to be spreadable if for any $k<n$, we have 
\begin{equation}\label{1}
(\xi_1,...,\xi_k)\overset{d}{=}(\xi_{l_1},...,\xi_{l_k}),\,\,\,\, l_1<l_2<\cdots<l_k
\end{equation}
An infinite sequence of random variables is said to be exchangeable or spreadable if all its finite subsequences have this property. 
In the study of distributional symmetries in classical probability, one of the most important results  is  de Finetti's  theorem which states that an infinite sequence of random variables, whose joint distribution is invariant under all finite permutations, is conditionally independent and identically distributed. Later,  in \cite{RN},  Ryll-Nardzewski showed that  de Finetti theorem hold under the weaker condition of spreadability. Therefore, for infinite sequences of random variables in classical probability, spreadability is equivalent to exchangeability. 

Recently, K\"oslter \cite{Ko} studied   three  kinds of distributional symmetries, which are  stationarity, contractability and exchangeablity, in noncommutative probability.  It was shown that exchangeability and spreadability do not characterize any universal independent relation in his framework.  In addition,  for infinite sequences, exchangeability is strictly stronger than spreadability in noncommutative probability.  It should be pointed out that the framework in his paper is a $W^*$-probability space with a faithful state.
In this paper, we will consider our problems in a more general framework.

In the 1980's, Voiculescu developed his free probability theory and introduced a universal independent relation, namely free independence, via reduced free products of unital $C^*$-algebras\cite{Vo1}. For more details of free probability, the reader is referred to the monograph \cite{VDN}. One can see that there is a deep parallel between classical probability and free probability.  Recently,  in \cite{KS}, K\"oslter and Speicher extended this parallel to the aspect of  distributional symmetries. In their work,  by strengthening classical exchangeability to quantum exchangeability, they proved a de Finetti type theorem for free independence, i.e.
for an infinite sequence of random variables, quantum exchangeability is equivalent to the fact that the the random variables are identically distributed and free with respect to the conditional expectation onto the tail algebra. The notion of quantum exchangeability is given by invariance conditions associated with  quantum permutation groups $A_s(n)$ of Wang \cite{Wan2}.    This noncommutative  de Finetti type theorem  is an instance  that free independence plays in the noncommutative world the same role as classical independence plays in the commutative world.  It naturally raises a motivation for further study of noncommutative symmetries that \lq\lq any result in classical probability should have an extension in free probability."  For applications of this philosophy, see\cite{BCS}, \cite{Cu3}, \cite{CS}.  Especially, in \cite{Cu3},  Curran introduced a quantum version of spreadability for free independence. It was shown that quantum spreadability is weaker than quantum exchangeability and is a characterization of free independence.  More specifically, in a $W^*$-probability space with a tracial faithful state, for an infinite sequence of random variables, quantum spreadability is equivalent to the fact that the the random variables are identically distributed and free with respect to the conditional expectation onto the tail algebra. In other words, quantum spreadability is equivalent to quantum exchangeability  for infinite sequences of random variables in tracial $W^*$-probability spaces.  Another remarkable application of quantum exchangeability was given by Freslon and  Weber \cite{FW}. They characterize Voiculescu's Bi-freeness \cite{Vo}  via certain invariance conditions associated with Wang's quantum groups $A_s(n)$.

%{Besides classical independence and free independence, in noncommutative probability, we have many other independences e.g. monotone independence \cite{Mu1}, boolean independence \cite{SW}, type B independence \cite{BGN} and more recently two-face freeness for pairs of random variables \cite{Vo}.} 
 
 In \cite{SW},  Speicher and Woroudi introduced another independence relation which is called boolean independence. It was show that boolean independence is related to full free product of algebras  \cite{CF1} and boolean product is the unique non-unital universal product in noncommutative probability \cite{Sp}.  The study of distributional symmetries for boolean independence was started  in \cite{Liu}.  We constructed a family of quantum semigroups in analogue with Wang's quantum permutation groups and defined their coactions on joint distributions of sequences. It  was shown that the distributional symmetries associated those coactions can be used to characterize  boolean independence in a proper framework.  For more details about boolean independence and universal products, see  \cite{Sp}.   It  inspires  us to study more distributional symmetries for boolean independence under the philosophy \lq\lq any result in classical probability and free probability should have an extension for boolean independence\rq\rq.  In analogue with easy quantum groups in \cite{BCS2}, we construct \lq\lq easy \rq\rq boolean semigroups  and study their de Finetti type theorems in \cite{Liu1}.  To apply our philosophy further,  it is naturally to find an extended de Finetti type theorem for boolean independence. Specifically, we need to find the \lq\lq noncommutative version of spreadability \rq\rq for boolean independence and prove an extended de Finetti type theorem associated with the noncommutative spreadability.\\

The main purpose of  this paper is to study noncommutative versions of spreadability and extended de Finetti type theorems associated with them.\\

Some other objects come into our consideration when we study spreadable sequences of random objects.  It was shown in \cite{Mu}, there are two other universal products in noncommutative probability if people do not require the universal construction to be commutative.  We call the two universal products monotone and anti-monotone product. As tensor product, free product and boolean product, we can define monotone and anti-monotone independence associated with monotone and anti-monotone product. Monotone independence and anti-monotone independence are essentially the same but with different orders, i.e. if $a$ is monotone with $b$, then $b$ is anti-monotone with $a$. For more details of monotone independence, the reader is referred to \cite{Mu1}, \cite{Po}.  It is well known that a sequence of monotone  random variables is not exchangeable but spreadable.  Therefore, there should be a noncommutative spreadability which can characterize conditionally monotone independence. 

 The first several sections devote to defining noncommutative distributional symmetries in analogue with spreadability and partial exchangeability.  Recall that in \cite{BCS} \cite{Cu3},  noncommutative distributional symmetries are defined via invariance conditions associated with certain quantum structures.  For instance, Curran's quantum spreadability is described by a family of quantum increasing sequences and their quantum family of maps in sense of Soltan. The family of quantum increasing sequences  are universal $C^*$-algebras $A_i(n,k)$ generated by the entries of a $n\times k$ matrix which satisfy certain relations $R$.  Following the idea in \cite{Liu}, to construct  a boolean type of spaces of  increasing sequences $B_i(n,k)$, we replace the unit partition condition in $R$  by an invariant projection condition.  Recall that in  In \cite{Fr1}, Franz studied  relations between freeness, monotone independence and boolean independence  via Bo{\.z}ejko, Marek and Speicher's two-state free products\cite{BS}. In his construction,  monotone product is something \lq\lq between" free product and boolean  product.  Thereby, we  construct the noncommutative spreadability for monotone independence by modifying quantum spreadability and our boolean spreadability. We will study simple relations between those distributional symmetries, i.e. which one is stronger.

As the situation for boolean independence, there is no nontrivial pair of monotonically independent random variables in $W^*$-probability spaces with faithful states.  Therefore, the framework we use in this paper is a $W*$-probability space with a non-degenerated normal state which gives  a faithful GNS representation of the probability space. In this framework, we will see that spreadability is too weak to ensure the existence of a conditional expectation. Recall that, in $W^*$-probability spaces with faithful states, we can define a normal shift on a unilateral infinite sequence of spreadable random variables.  Here, \lq\lq unilateral\rq\rq means the sequence is indexed by natural numbers $\mathbb{N}$. An important property of this shift is that its norm is one. Therefore, given an operator, we can construct a WOT convergent sequence of bounded variables via shifts. The is the key step to construct a normal conditional expectation in previous works.   But, in $W^*$-probability spaces with non-degenerated normal states, the unilateral shift of spreadable random variables is not necessarily norm one. An example is provided in the beginning of section  6. Actually, the sequence of random variables are monotonically spreadable which is an invariance condition stronger than classical spreadability.  Therefore,  we can not construct a conditional expectation, for unilateral sequences, via shifts under the condition of spreadability. To fix this issue,  we will consider bilateral sequences of random variables instead of  unilateral sequences.  \lq\lq bilateral\rq\rq means that the sequences are indexed by integers $\mathbb{Z}$.   In this framework, we will see that the shift of spreadable random variables is norm one so that we can define a conditional expectation via shifts by following K\"ostler's construction.  Notice that the index set $\mathbb{Z}$ has two infinities, i.e. the positive infinity and the negative infinity.  Therefore, we will have two tail algebras with  respect to the two infinities and will define two conditional expectations consequently. We denote by $E^+$  the conditional expectation which shifts  indices to the positive infinity and $E^-$ the conditional expectation which shifts  indices to negative infinity.   We will see that the two tail algebras are subsets of fixed points of the shift and the conditional expectations may not be extended normally to the whole algebra.  In general,  the two tail algebras are different and the conditional expectation may have different properties.  To noncommutative spreadability for monotone independence, we have the following:

\begin{theorem}\label{theorem 1}
Let $(\A,\phi)$ be a non degenerated $W^*$-probability space and $(x_i)_{i\in\mathbb{Z}}$ be a bilateral infinite sequence of selfadjoint random variables which generate $\A$ as a von Neumann algebra. Let $\A^+_k$ be the WOT closure of the non-unital algebra generated by $\{x_i|i\geq k\}$. Then the following are equivalent:
\begin{itemize}
\item[a)] The joint distribution of $(x_i)_{i\in \mathbb{Z}}$  is monotonically spreadable. 
\item[b)]  For all $k\in\mathbb{Z}$, there exits a $\phi$ preserving conditional expectation $E_k:\A^+_k\rightarrow \ATP$ such that the sequence  $(x_i)_{i\geq k}$ is identically distributed and monotonically independent with respect $E_k$.  Moreover, $E_k|_{\A_{k'}}=E_{k'}$ when $k\geq k'$.
\end{itemize}
\end{theorem}

In general, we can not extend $E^+$ to the whole algebra $\A$, but we have the following:
\begin{proposition}
Let $(\A,\phi)$ be a non degenerated $W^*$-probability space and $(x_i)_{i\in\mathbb{Z}}$ be a bilateral infinite sequence of selfadjoint random variables which generate $\A$ as a von Neumann algebra. If the joint distribution of $(x_i)_{i\in \mathbb{Z}}$  is monotonically spreadable, then $E^-$ can be extend to the whole algebra $\A$ normally.
\end{proposition}

We will see that boolean spreadability implies monotone spreadability and anti-monotone spreadability.  Therefore, both $E^+$ and $E^-$ can be extended normally to the whole algebra $\A$. Moreover, for boolean spreadable sequences, $E^+=E^-$ and the two algebras are identical. In summary, we have 
\begin{theorem}\label{theorem 2}
Let $(\A,\phi)$ be a non degenerated $W^*$-probability space and $(x_i)_{i\in\mathbb{Z}}$ be a bilateral infinite sequence of selfadjoint random variables which generate $\A$ as a von Neumann algebra. Then the following are equivalent:
\begin{itemize}
\item[a)] The joint distribution of $(x_i)_{i\in \mathbb{N}}$  is boolean spreadable. 
\item[b)]  The sequence $(x_i)_{i\in\mathbb{Z}}$ is identically distributed and boolean independent with respect to the $\phi-$preserving conditional expectation $E$ onto the non unital tail algebra of the $(x_i)_{i\in\mathbb{Z}}$
\end{itemize}
\end{theorem}

The paper is organized as follows:  In section 2, we will introduce preliminaries and notation from noncommutative probability and recall Wang's quantum permutation groups and boolean quantum semigroups. In section 3, we briefly review distributional symmetries for finite sequences of random variables in classical probability and we restate these symmetries in  words of quantum maps. Then, we introduce noncommutative versions of these symmetries and their quantum maps.  In the end of this section, we will define quantum spreadability, monotone spreadability and boolean spreadability for bilateral infinite sequences of random variables.  In section 4, we will study simple relations between our noncommutative symmetries.  In particular, we will show boolean exchangeability is strictly stronger than boolean spreadability. Therefore, operator-valued boolean independent random variables are boolean spreadable.  In section 5, we will introduce an equivalence relation on the set of sequences of indices. With the help of the equivalence relation, we will show that  operator-valued monotone independent sequences of random variables are monotonically spreadable. In section 6, we first provide an example that a monotonically spreadable unilateral sequence of bounded random variables is unbounded.  Therefore, we cannot define conditional expectation for unilateral spreadable sequences via shifts in a $W^*$-probability space with a non-degenerated normal state. Then we will turn to study  bilateral sequences of random variables.  We will introduce tail algebras associated with positive infinity and negative infinity and study elementary properties of conditional expectations associated with the two tail algebras. In section 7, we will study properties of conditional expectations under the assumption that our bilateral sequences are monotonically spreadable. In section 8, we will prove a Ryll-Nardzewski type theorem for Monotone independence. In section 9, we will prove a Ryll-Nardzewski type theorem for boolean independence.

\section{Preliminaries and examples}

We recall some necessary definitions and notions from noncommutative probability. For further details, see contexts \cite{KS}, \cite{NS}, \cite{VDN}, \cite{Po}.
\begin{definition}\normalfont  A non-commutative probability space $(\A,\phi)$ consists of a unital algebra $\A$ and a linear functional $\phi:\A\rightarrow \C$ such that $\phi(1_{\A})=1$. $(\A,\phi)$ is  called a $*$-probability space if $\A$ is a $*$-algebra and $\phi(xx^*)\geq 0$ for all $x\in\A$. $(\A,\phi)$ is called a  $W^*$- probability space if $\A$ is a $W^*$-algebra and $\phi$ is a normal state on it.  We will not assume that $\phi$ is faithful.  The elements of $\A$ are called  random variables. Let $x\in\A$ be a random variable, then its distribution is a linear functional $\mu_x$ on $\C[X]$( the algebra of complex polynomials in one variable), defined by $\mu_x(P)=\phi(P(x))$.
\end{definition}

\begin{definition} Let  $\A$ be a  $W^*$-algebra, a normal state $\phi$ on $\A$ is said to be non-degenerated if  $x=0$ whenever $\phi(axb)=0$ for all $a,b\in \A.$  
\end{definition}

 By proposition 7.1.15 in \cite{KR},  if $\phi$ is a non-degenerated normal state on $\A$ then the GNS representation associated to $\phi$ is faithful. In this paper, we will work with $W^*$-probability space with  a non-degenerated normal state. The reason is that there is no non-trivial pair of boolean or monotonically independent random variables in $W^*$-probability spaces with faithful states. See \cite{Liu}.

\begin{definition}\normalfont Let $I$ be an index set. The algebra of noncommutative polynomials in $|I|$ variables, $\C\langle X_i|i\in I\rangle$, is the  linear span of $1$ and noncommutative monomials of the form $X^{k_1}_{i_1}X^{k_2}_{i_2}\cdots X^{k_n}_{i_n}$ with $i_1\neq i_2\neq\cdots \neq i_n\in I$ and all $k_j$'s are positive integers. For convenience, we will denote by $\C\langle X_i|i\in I \rangle_0$  the set of noncommutative polynomials without a constant term.
Let $(x_i)_{i\in I}$ be a family of random variables in a noncommutative probability space $(\A,\phi)$. Their  joint distribution is the linear functional $\mu:\C\langle X_i|i\in I \rangle\rightarrow \C$ defined by
$$\mu(X^{k_1}_{i_1}X^{k_2}_{i_2}\cdots X^{k_n}_{i_n})=\phi(x^{k_1}_{i_1}x^{k_2}_{i_2}\cdots x^{k_n}_{i_n}),$$
and $\mu(1)=1$.
\end{definition}

In general, the joint distribution depends on the order of the random variables, e.g 
$ \mu_{x,y}$ may not equal $\mu_{y,x}$. According to our notation, $\mu_{x,y}(X_1X_2)=\phi(xy)$, but $\mu_{y,x}(X_1X_2)=\phi(yx)$.  In this paper, our index set $I$ is always an ordered set with order \lq\lq$>$ \rq\rq e.g. $\mathbb{N}$, $\mathbb{Z}$.

\begin{definition}\normalfont Let  $(\A,\phi)$ be a noncommutative probability space. A family of (not necessarily unital) subalgebras $\{\A_i|i\in I \}$ of $\A$ is said to be boolean independent if
$$\phi(x_1x_2\cdots x_n)= \phi(x_1)\phi(x_2)\cdots \phi(x_n)$$
whenever $x_k\in\A_{i(k)}$ with $i(1)\neq i(2)\neq\cdots \neq i(n)$. The family of subalgebras $\{\A_i|i\in I \}$  is said to be monotonically independent if 
$$\phi(x_1\cdots x_{k-1}x_kx_{k+1}\cdots x_n)=\phi(x_k)\phi(x_1\cdots x_{k-1}x_{k+1}\cdots x_n)$$
whenever $x_j\in\A_{i_j}$ with $i_1\neq i_2\neq\cdots \neq i_n$ and $i_{k-1}<i_k>i_{k+1}$. A set of random variables $\{x_i\in\A|i\in I\}$ is said to be boolean(monotonically) independent if the family of non-unital subalgebras $\A_i$, which are generated by $x_i$ respectively, is boolean(monotonically) independent.
\end{definition}
One refers to \cite{Fr} for more details of boolean product and monotone product of random variables. In general, the framework for boolean independence and monotone independence is a non-unital algebra.  Thereby, we will use the following version of operator valued probability spaces: 
\begin{definition}\normalfont {\bf Operator valued probability space} An operator valued probability space $(\A,\B, E:\A\rightarrow \B)$ consists of an algebra  $\A$, a subalgebra $\B$ of $\A$ and a $\B-\B$ bimodule linear map $E:\A\rightarrow \B$, i.e.
 $$E[b_1ab_2]=b_1E[a]b_2,\,\,\, E[b]=b$$
for all $b_1,b_2,b\in\B$ and $a\in\A$. According to the definition in \cite{St}, we call $E$ a conditional expectation from $\A$ to $\B$ if $E$ is onto, i.e. $E[\A]=\B$. The elements of $\A$ are called random variables.
\end{definition}
\begin{remark} In free probability theory, $\A$ and $\B$ are assumed to be unital and share the same unit
\end{remark}

\begin{definition}\normalfont For an algebra $\B$, we denote by $\B\langle X\rangle$  the algebra which is freely generated by $\B$ and the indeterminant $X$. Let $1_X$ be the identity of $\C\langle X\rangle$, then $\B\langle X\rangle$ is set of linear combinations of the elements in $\B$ and the noncommutative monomials $b_0Xb_1Xb_2\cdots b_{n-1}Xb_n$ where $b_k\in\B\cup \{\C 1_X\}$ and $n\geq 0$. The elements in  $\B\langle X\rangle$ are called $\B$-polynomials. In addition,  $\B\langle X\rangle_0$ denotes the subalgebra of $\B\langle X\rangle$ which does not contain a constant term  in $\B$, i.e. the linear span of the noncommutative monomials $b_0Xb_1Xb_2\cdots b_{n-1}Xb_n$ where $b_k\in\B\cup \{\C 1_X\}$ and $n\geq 1$. 
\end{definition}

Here are the operator valued versions of noncommutative independences:

\begin{definition}\normalfont
Let  $\{x_i\}_{i\in I}$  be a family of random variables in an operator valued probability space $(\A,\B,E:\A\rightarrow \B)$, where $\A$ and $\B$ are not necessarily unital.  $\{x_i\}_{i\in I}$ is  said to be boolean independent over $\B$ if 
$$E[p_1(x_{i_1})p_2(x_{i_2})\cdots p_n(x_{i_n})]=E[p_1(x_{i_1})]E[p_2(x_{i_2})]\cdots E[p_n(x_{i_n})]$$
whenever $i_1,\cdots,i_n\in I$, $i_1\neq i_2\neq\cdots\neq i_n$ and $p_1,\cdots,p_n\in\B\langle X\rangle_0$.\\
$\{x_i\}_{i\in I}$ is  said to be monotonically independent over $\B$ if 
$$\begin{array}{rcl}
&&E[p_1(x_{i_1})\cdots p_{k-1}(x_{i_{k-1}})p_{k}(x_{i_{k}})p_{k+1}(x_{i_{k+1}})\cdots p_n(x_{i_n})]\\&=&E[p_1(x_{i_1})\cdots p_{k-1}(x_{i_{k-1}})E[p_{k}(x_{i_{k}})]p_{k+1}(x_{i_{k+1}})\cdots p_n(x_{i_n})]
\end{array}$$
whenever $i_1,\cdots,i_n\in I$, $i_1\neq i_2\neq\cdots\neq i_n$, $i_{k-1}<i_k>i_{k+1}$ and $p_1,\cdots,p_n\in\B\langle X\rangle_0$.\\
\end{definition}
Notice that there is another natural order \lq\lq$<$ \rq\rq on $I$, i.e. $a<b$ if $b>a$. Therefore, we can define another noncommutative independence relation.  $\{x_i\}_{i\in I}$ is  said to be anti-monotonically independent with respect to $E$ and index order \lq\lq$>$" if  $\{x_i\}_{i\in I}$ is  said to be anti-monotonically independent with respect to $E$ and index order \lq\lq$<$\rq\rq.  See more details in \cite{Mu}.

\subsection{Noncommutative distributional symmetries}
 Recall that, in \cite{Wan}, Wang introduced the following quantum analogue of permutation groups:
\begin{definition} \normalfont $A_s(n)$ is defined as the universal unital $C^*$-algebra generated by elements $(u_{i,j})_{i,j=1,\cdots n}$ such that we have
\begin{itemize}
\item Each $u_{i,j}$ is an orthogonal projection, i.e. $u_{ij}^*=u_{ij}=u_{ij}^2$ for all $i,j=1,\cdots, n$.
\item The elements in each row and column of $u=(u_{ij})_{i,j=1,\cdots,n}^n$ form a partition of unit, i.e. are orthogonal and sum up to 1: for each $i=1,\cdots,n$ and $k\neq l$ we have
$$
\begin{array}{rcl}
u_{ik}u_{il}=0&\text{and}& u_{ki}u_{li}=0;\\
\end{array}.$$
\end{itemize}
and for each $i=1,\cdots, n$ we have
$$\sum\limits_{k=1}^nu_{ik}=1=\sum\limits_{k=1}^nu_{ki}.$$
\end{definition}

$A_s(n)$ is a compact quantum group in sense of Woronowicz \cite{Wo2}, with comultiplication, counit and antipode given by the formulas:
$$\Delta u_{i,j}=\sum\limits_{k=1}^n u_{i,k}\otimes u_{k,j}$$
$$\epsilon (u_{i,j})=\delta_{i,j} $$
$$S(u_{i,j})=u_{j,i}$$
It was shown that the this quantum structure can be used to characterize conditionally free independence \cite{KS}.\\

In \cite{Liu}, we modify the universal conditions of Wang's quantum permutation groups: By replacing the condition associated with partitions of the unit by a condition associated with an invariant projection, we get the following universal algebras:

\noindent{\bf Quantum semigroups ($B_s(n)$, $\Delta$):} The algebra $B_s(n)$ is defined as the universal unital $C^*$-algebra generated by elements $u_{i,j}$ $(i,j=1,\cdots n)$ and a projection ${\bf P}$ such that we have
\begin{itemize}
\item each $u_{i,j}$ is an orthogonal projection, i.e. $u_{i,j}^*=u_{i,j}=u_{i,j}^2$ for all $i,j=1,\cdots, n$.
\item $$
\begin{array}{rcl}
u_{i,k}u_{i,l}=0&\text{and}& u_{k,i}u_{l,i}=0\\
\end{array}
$$
whenever $k\neq l$.
\item For all $1\leq i\leq n$,  ${\bf P}=\sum\limits_{k=1}^{n}u_{k,i}{\bf P}$.
\end{itemize}
There is a natural comultiplication $\Delta:B_s(n)\rightarrow B_s(n)\otimes_{min} B_s(n)$ defined by
$$\Delta u_{i,j}=\sum\limits_{k=1}^n u_{i,k}\otimes u_{k,j},\,\,\, \Delta(\p)=\p\otimes\p,\,\,\, \Delta(I)=I\otimes I,$$
where $\otimes_{min}$ stands for the reduced $C^*$-tensor product. The existence of of these maps is given by the universal property of $B_s(n)$.  Therefore, $(B_s(n),\Delta)$'s are  quantum semigroups in sense of So{\l}tan \cite{So}. These quantum structures can characterize conditionally boolean independence, see more details in \cite{Liu}.

\section{Distributional symmetries for finite sequences of random variables}
In this section, we will review  two kinds of distributional symmetries which are spreadability and partial exchangeability, in classical probability.  In \cite{Ka}, we see that  the distributional symmetries can be defined for either finite sequences or infinite sequences.   Moreover,  each kind of  distributional symmetry for  infinite sequences of random objects is  determined by  distributional symmetries on all its finite subsequences.  For example, an infinite sequence of random variables is exchangeable iff all its finite subsequences are exchangeable.  We will present distributional symmetries for finite sequences and  then introduce their counterparts in noncommutative probability.   In the first subsection, we recall  notions of  spreadability and partial exchangeability  in classical probability and rephrase these notions in words of quantum maps.  In the second subsection, we will introduce  counterparts of spreadability and partial exchangeability in noncommutative probability.  Even though there are many interesting properties of partial exchangeability, we are not going to study it too much in this paper because the main problem we concern is about extended de Finetti  type theorems for  noncommutative spreadable sequences of random variables.   We will discuss relations between those noncommutative distributional symmetries in the next section.

\subsection{Spreadability and partial exchangeability}

Recall that in \cite{Ka1}, a finite sequence of random variables $(x_1,...,x_n)$ is said to be spreadable if for any $k<n$, we have 
\begin{equation}\label{1}
(x_1,...,x_k)\overset{d}{=}(x_{l_1},...,x_{l_k}),\,\,\,\, l_1<l_2<\cdots<l_k
\end{equation}
For fixed natural numbers $n>k$, it is mentioned in \cite{Cu3},  the above relation can be described in  words of  quantum family of maps in  sense of Soltan \cite{So1}:  Considering the space $I_{k,n}$ of increasing sequences $\I=(1\leq i_1<\cdots<i_k\leq n)$. For $1\leq i\leq n$,  $1\leq j\leq k$, define $f_{i,j}:I_{k,n}\rightarrow \C$ by:
$$f_{i,j}(\I)=\left\{\begin{array}{cc}
1, & i_j=i\\
0,&i_j\neq i
\end{array}\right..
$$ 
If we consider $I_{n,k}$ as a discrete space, then the functions $f_{i,j}$ generate $C(I_{n,k})$ by the Stone-Weierstrass theorem.  Let $\C[X_1,...,X_m]$ be the set of commutative polynomials in $m$ variables. The algebra $C(I_{n,k})$ together with an algebraic  homomorphism $\alpha:\C[X_1,...,X_k]\rightarrow \C[X_1,...,X_n]\otimes C(I_{k,n})$ define by:
$$\alpha: X_{j}=\sum\limits_{i=1}^n X_i\otimes f_{i,j}, \,\,\,\, \alpha(1)=1_{C(I_{k,n})} $$
defines a quantum family of maps from $\{1,...,k\}$ to $\{1,...,n\}$.

We can use this family of quantum maps to rephrase  equation (\ref{1}):  Let $\mu_{x_1,....,x_n}$ be the joint distribution of $(x_1,...,x_n)$. For fixed natural numbers $n>k$,  
\begin{equation}\label{3}
\mu_{x_1,...,x_k}(p)1_{C(I_{n,k})}=\mu_{x_1,...,x_n}\otimes id_{C(I_{n,k})}(\alpha(p))
\end{equation}
for all $p\in\C[x_1,...,x_k]$.\\
For completeness, we provide a sketch of proof here:
Suppose equation (\ref{1}) holds. Let $p=X_{j_1}^{i_1}\cdots X_{j_m}^{i_m}$ be a monomial in $\C[X_1,...,X_k]$ such that $1\leq j_1<j_2<\cdots <j_m\leq k$ and $i_1,...,i_m$ are positive integers.  Let $\I=(1\leq l_1<\cdots<l_k\leq n)$ be a point in 
$I_{k,n}$.  Then, the $\I$-th component of $\mu_{x_1,...,x_k}(p)1_{C(I_{n,k})}$ is $E[x_{j_1}^{i_1}\cdots x_{j_m}^{i_m}]$.  The $\I$-th component$\mu_{x_1,...,x_n}\otimes id_{C(I_{n,k})}(\alpha(p))$ is 
$$
\sum\limits_{s_1,...,s_m=1}^n E[x_{s_1}^{i_1}\cdots x_{s_m}^{i_m}](f_{s_1,j_1}\cdots f_{s_m,j_m})(\I).
$$
According the definition of $f_{i,j}$, $(f_{s_1,j_1}\cdots f_{s_m,j_m})(\I)$ is not vanished only if $s_t=l_{j_t}$ for all $1\leq t\leq m$. Therefore, 
$$
\sum\limits_{s_1,...,s_m=1}^n E[x_{s_1}^{i_1}\cdots x_{s_l}^{i_l}](f_{s_1,j_1}\cdots f_{s_m,j_m})(\I)=E[x_{l_{j_1}}^{i_1}\cdots x_{l_{j_m}}^{i_m}].
$$
Since $1\leq j_1<j_2<\cdots <j_m\leq k$ and $\I$ is an in creasing sequence, we have $1\leq l_{j_1}<\cdots<l_{j_{m}}\leq n$. Hence, the $\I$-components of the  two sides of equation (\ref{3}) are equal to each other. Since $\I$ is arbitrary, equation (\ref{3}) holds.  By checking the $\I$ component of equation \ref{3}, we can also show that (\ref{3}) implies (\ref{1}).  We will say that  $(\xi_1,...,\xi_n)$ is $(n,k)$-spreadable if $(x_1,...,x_n)$ satisfies  equation (\ref{3}).

\begin{remark}\normalfont
We see that the above $(n,k)$-spreadability describes limited relations between the mixed moments of $(x_1,...,x_n)$. Once we fix $n,k$, the $(n,k)$-spreadability  gives no information about mixed moments which involve $k+1$ variables. For example, let $n=4$, $k=2$ and assume that $(x_1,...,x_4)$ is a $(4,2)$-spreadable sequence. According to equation (\ref{1}), we  know nothing  about the relation between $E[x_1 x_2x_3]$ and $E[x_2x_3x_4]$.  We will call this  kind of distributional symmetries  partial symmetries because they just provide information of  part of mixed moments but not all.
\end{remark}

By using the idea of partial symmetries, we can define another family of distributional symmetries which is stronger than $(n,k)$-spreadability but weaker than exchangeability.  
\begin{definition}
For fixed natural numbers $n>k$, we say a sequence of random variables $(x_1,...,x_n)$ is $(n,k)$-exchangeable if 
$$(x_1,...,x_k)\overset{d}{=}(x_{\sigma(1)},...,x_{\sigma(k)}),\,\,\,\, \forall \sigma\in S_n,$$
where $S_n$ is the permutation group of n elements.
\end{definition}
This kind of exchangeability is called partial exchangeability. For more details, see \cite{DF}. As well as $(n,k)$-spreadability, we can rephrase partial exchangeability in words of quantum family of maps:  Considering the space $E_{n,k}$ of  length $k$ sequences $\{\I=(i_1,...,i_k)| 1\leq i_1,..., i_k\leq n, i_j\neq i_{j'}\,\, \text{for}\,\, j\neq j'\}$. For $1\leq i\leq n$,  $1\leq j\leq k$, define $g_{i,j}:I_{n,k}\rightarrow \C$ by:
$$g_{i,j}(\I)=\left\{\begin{array}{cc}
1, & i_j=i,\\
0,&i_j\neq i.
\end{array}\right.
$$ 
Given two different sequences $\I=(i_1,...,i_k)$ and  $\I'=(i'_1,...,i'_{k})$,  there must exists a number $j$ such that $i_j\neq i'_j$. Then, we have that $g_{i,i_j}(\I)=1\neq 0=g_{i,i_j}(\I)$. Therefore, the set of functions $\{g_{i,j}|i=1,...,n;j=1,...,k\}$ separates $E_{n,k}$.  According to Stone Weierstrass theorem, the functions $g_{i,j}$  generate $C(E_{n,k})$. Again, we can define a homomorphism $\alpha':\C[X_1,...,X_k]\rightarrow \C[X_1,...,X_n]\otimes C(E_{n,k})$ by the following formulas:
$$\alpha': X_{j}=\sum\limits_{i=1}^n X_i\otimes g_{i,j}, \,\,\,\, \alpha'(1)=1_{C(I_{k,n})}. $$
\begin{lemma}
Let $\mu_{x_1,....,x_n}$ be the joint distribution of $x_1,...,x_n$.  Then 
$$\mu_{x_1,...,x_k}(p)1_{C(I_{n,k})}=\mu_{x_1,...,x_n}\otimes id_{C(I_{n,k})}(\alpha(p))$$
for all $p\in\C[X_1,...,X_k]$
if and only if $x_1,...,x_n$ is $(n,k)$ exchangeable.
\end{lemma}
The proof is similar the proof of $(n,k)$-spreadability, we just need to check the values at all components of $E_{n,k}$.

\subsection{Noncommutative analogue of partial symmetries} 
Now, we turn to introduce noncommutative versions of spreadability and partial exchangeability.  The pioneering  work was done by Curran \cite{Cu3}.  He defined a quantum version of $C(I_{n,k})$ in analogue of Wang's quantum permutation groups as following: 
\begin{definition}\normalfont
For $k,n\in\mathbb{N}$ with $k\leq n$, the quantum increasing space $A(n,k)$ is the universal unital $C^*-$algebra generated by elements $\{u_{i,j}|1\leq i\leq n, 1\leq j\leq k\}$ such that
\begin{itemize}
\item[1.] Each  $u_{i,j}$ is an orthogonal projection: $u_{i,j}=u_{i,j}^*=u_{i,j}^2$ for all $i=1,...,n;j=1,...,k$.
\item[2.] Each column of the rectangular matrix $u=(u_{i,j})_{i=1,...,n;j=1,...,k}$ forms a partition of unity: for $1\leq j\leq k$ we have $\sum\limits_{i=1}^nu_{i,j}=1$.
\item[3.] Increasing sequence condition: $u_{i,j}u_{i',j'}=0$ if $j<j'$ and $i\geq i'$.
\end{itemize}
\end{definition}

\begin{remark}\normalfont
Our notation is different from Curran's,  we use $A_i(n,k)$ instead of his $A_i(k,n)$ for our convenience. 
\end{remark}

For any natural numbers $k<n$, in analogue of coactions of $A_s(n)$,  there is a unital $*$-homomorphism $\alpha_{n,k}:\C\langle X_1,...,X_k\rangle\rightarrow \C\langle X_1,...,X_n \rangle\otimes A_i(n,k)$ determined by:
$$ \alpha_{n,k}(X_j)=\sum\limits_{i=1}^n X_i\otimes u_{i,j}.$$

The quantum spreadability of random variables is defined as the following:
\begin{definition}\normalfont
Let $(\A,\phi)$ be a noncommutative probability space. A finite ordered sequence of random variables $(x_i)_{i=1,...,n}$ in $\A$ is said to be $A_i(n,k)$-spreadable if their joint distribution $\mu_{x_1,...,x_n}$ satisfies:
   $$\mu_{x_1,...,x_n}(p)1_{A_i(n,k)}=\mu\otimes id_{A_i(n,k)}(\alpha_{n,k}(p)),$$ 
for all $p\in \C\langle X_1,...,X_k\rangle$.  $(x_i)_{i=1,...,n}$ is said to be  quantum spreadable  if $(x_i)_{i=1,...,n}$  is  $A_i(n,k)$-spreadable for all $k=1,...,n-1$.
\end{definition}

\begin{remark}
In \cite{Cu3}, Curran studied sequences of $C^*$-homomorphisms  which are more general than random variables. For consistency, we  state his definitions in words of random variables.    It is a routine to  extend our work to the framework of sequences of $C^*$-homomorphisms.
\end{remark}

Recall that in \cite{Liu}, by replacing the condition associated with partitions of the unity of  Wang's quantum permutation groups, we defined a family of quantum semigroups with invariant projections. With a natural family of  coactions, we defined  invariance conditions which can characterize conditional boolean independence.  Here,  we can modify Curran's quantum increasing spaces in the same way:

\begin{definition}\normalfont
For $k,n\in\mathbb{N}$ with $k\leq n$, the noncommutative increasing space $B_i(k,n)$ is the unital universal  $C^*-$algebra generated by elements $\{\ub_{i,j}|1\leq i\leq n, 1\leq j\leq k\}$ and an invariant projection $\p$ such that
\begin{itemize}
\item[1.] Each $\ub_{i,j}$ is an orthogonal projection:$\ub_{i,j}=(\ub_{i,j})^*=(\ub_{i,j})^2$ for all $i=1,...,n;j=1,...,k$.
\item[2.] For $1\leq j\leq k$ we have $\sum\limits_{i=1}^n\ub_{i,j}\p=\p$.
\item[3.] Increasing sequence condition: $\ub_{i,j}\ub_{i',j'}=0$ if $j<j'$ and $i\geq i'$.
\end{itemize}
\end{definition}

The same as $A_i(n,k) $,  there is a unital $*$-homomorphism $\alpha^{(b)}_{n,k}:\C\langle X_1,...,X_k\rangle\rightarrow \C\langle X_1,...,X_n \rangle\otimes B_i(n,k)$ determined by:
$$ \alpha^{(b)}_{n,k}(x_j)=\sum\limits_{i=1}^n x_i\otimes \ub_{i,j}$$

As  boolean exchangeability defined in \cite{Liu},  we have
\begin{definition}\normalfont
A finite ordered sequence of random variables $(x_i)_{i=1,...,n}$ in $(\A,\phi)$ is said to be $B_i(n,k)$-spreadable if their joint distribution $\mu_{x_1,...,x_n}$ satisfies:
   $$\mu_{x_1,...,x_n}(p)\p=\p\mu\otimes id_{B_i(n,k)}(\alpha^{(b)}_{n,k}(p))\p,$$ 
for all $p\in \C\langle X_1,...,X_k\rangle$.  $(x_i)_{i=1,...,n}$ is said to be  boolean spreadable  if $(x_i)_{i=1,...,n}$  is  $B_i(n,k)$-spreadable for all $k=1,...,n-1$.
\end{definition}
We will see that $B_i(k,n)$ is an  increasing space of  boolean type, because we can derive an extended de Finetti  type theorem for boolean independence.

Recall that, in \cite{Fr1}, Franz showed some relations between free independence, monotone independence and boolean independence  via Bo{\.z}ejko, Marek and Speicher's two-states free products\cite{BS}. We can see that  monotone product is  \lq\lq between" free product and boolean  product.  From this viewpoint of Franz's work, we may hope to define a kind of \lq\lq spreadability\rq\rq for monotone independence by modifying quantum spreadability and boolean spreadability. Notice that there are at least two ways  to get quotient algebras of $B_i(k,n)$'s such that the $\p$-invariance condition of the quotient algebras is equivalent quantum spreadability: 
\begin{itemize}
\item[1.] Require $\p$ to be the unit of the algebra.
\item[2.] Let $P_j=\sum\limits_{i=1}^nu_{i,j}$, require $P_{j'}u_{ij}=u_{ij}P_{j'}$ for all $1\leq j,j'\leq k$ and $1\leq i\leq n$.
\end{itemize}
 To define the monotone increasing spaces,  we  modify the second condition a little:
\begin{definition}\normalfont
For fixed $n,k\in\mathbb{N}$ and $k<n$, a monotone increasing sequence space $M_i(n,k)$ is the universal unital $C^*$-algebra generated by elements $\{\um_{i,j}\}_{i=1,...,n; j=1,...,k}$
\begin{itemize}
\item[1.] Each $u_{i,j}$ is an orthogonal projection;
\item[2.] Monotone condition: Let $P_j=\sum\limits_{i=1}^n\um_{i,j}$, $P_j\um_{i'j'}=u_{i'j'}$ if $j'\leq j.$
\item[3.] $\sum\limits_{i=1}^n\um_{i,j}P_1=P_1$ for all $1\leq j\leq k.$
\item[4.] Increasing condition: $\um_{i,j}\um_{i',j'}=0$ if $j<j'$ and $i\geq i'.$

\end{itemize}
\end{definition}

We see that  $P_1$ plays the role as the invariant projection $\p$ in the boolean case. For consistency, we denote $P_1$ by $\p$.  Then, we can define a  $\p$-invariance condition associated with  $M_i(n,k)$ in analogy with $B_i(n,k) $: For fixed $n,k\in\mathbb{N}$ and $k<n$, there is a unique unital $*$- isomorphism $\am_{n,k}:\C\langle X_1,...,X_k\rangle\rightarrow \C\langle X_1,...,X_n\rangle\otimes M_i(n,k)$ such that
$$\am_{n,k}(X_j)=\sum\limits_{i=1}^n X_i\otimes\um_{i,j}.$$
The existence of such a homomorphism is given by the universality of $\C\langle X_1,...,X_k\rangle$.

\begin{definition}\normalfont
A finite ordered sequence of random variables $(x_i)_{i=1,...,n}$ in $(\A,\phi)$ is said to be $M_i(n,k)$-invariant if their joint distribution $\mu_{x_1,...,x_n}$ satisfies:
   $$\mu_{x_1,...,x_k}(p)\p=\p\mu_{x_1,...,x_n}\otimes id_{M_i(n,k)}(\am_{n,k}(p))\p,$$ 
for all $p\in \C\langle X_1,...,X_k\rangle$. $(x_i)_{i=1,...,n}$ is said to be monotonically spreadable if it is $M_i(n,k)$-invariant for all $k=1,...,n-1.$
\end{definition}
We will see that these invariance conditions can  characterize conditionally Monotone independence in a proper framework.

As remark 2.3 in \cite{Cu3}, a first question to our definitions is whether $A_i(n,k)$, $B_i(n,k)$, $M_i(n,k)$ exist.  In \cite{Cu3},  Curran has showed several nontrivial representations of $A_i(n,k)$.  In the following,  we provide a family of presentations of $A_i(n,k)$, $B_i(n,k)$, $M_i(n,k)$ for $n>k$:
Fix natural numbers $n>k$, let $l_1,...,l_k\in\mathbb{N}$ such that $$l_1+\cdots+l_k=n.$$ We denote by $\hh_{i}$ a $l_i$-dimensional  Hilbert spaces with orthonormal basis $\{e^{(i)}_{j}|j=1,...,l_i\}$. Let $I_{l_i}$ be the unit of  the algebra $B(\hh_{l_i})$, $P_{e^{(l_i)}_j}$ be the one dimensional orthogonal projection onto $\C e^{(l_i)}_j$,  $P_i$ be the one dimensional projection onto $\C\sum\limits_{j}e^{(l_i)}_j$. Consider the following matrix:
$$
     \left(\begin{array}{cccc}
     P_{1,1}   &0             &\cdots    &0\\
     \vdots&   \vdots       &\ddots    &\vdots\\
     P_{1,l_1}  &0            &     \cdots& 0\\
      0             &P_{2,1}   &\cdots      &0\\
     \vdots      & \vdots    &   \ddots   &\vdots\\
     0              &P_{2,l_2}& \cdots  &  0     \\
     0              &0            &\cdots & 0\\
     \vdots      & \vdots   &\ddots    &\vdots \\
     0              &0           & \cdots    &P_{k,1}\\
     \vdots      &\vdots    &\ddots    &\vdots\\
     0              &0            &\cdots    &P_{k,l_k}\\
     
     \end{array}\right).
$$
We see that  the entries of the matrix satisfy the increasing condition of spaces of increasing sequences. By choosing proper projections $P_{i,j}$, we will get representations for our universal algebras:\\
Quantum family of increasing sequences:  For each $1\leq j\leq k$, the algebra generated by $\{P_{e_j^{i}}|i=1,...,l_j\}$ is isomorphic to $C^*(\mathbb{Z}_{l_j})$. The reduced free product $*_{j=1}^k\mathbb{Z}_{l_i}$ is a quotient algebra of $A_i(n,k)$. One can define a $C^*$-homomorphism $\pi$ from $A_i(n,k)$, such that 
$$
\pi(u_{i,j})=\left\{\begin{array}{lcl}
\text{the image of}\,\, P_{e^{(l_i)}_{j'}}\,\,\text{in}\, *_{j=1}^kC^*(\mathbb{Z}_{l_j}) & \text{if} &  0<j'=j-\sum\limits_{l=m}^{i-1}l_m\leq l_i\\
0& \text{if} &  \text{otherwise}\\
\end{array}\right.
$$
Boolean family of increasing sequences: One can define a $C^*$-homomorphism $\pi$ from $B_i(n,k)$ into $B(\bigotimes\limits_{i=1}^k \hh_i)$ such that
$$
\pi(u_{i,j})=\left\{\begin{array}{lcl}
 \bigotimes\limits_{m_1=1}^{i-1} P_{l_{m_1}} \otimes P_{e^{(l_i)}_{j'}}\bigotimes\limits_{m_2=i+1}^{k} P_{l_{m_2}} & \text{if} &  0<j'=j-\sum\limits_{l=m}^{i-1}l_m\leq l_i\\
0& \text{if} &  \text{otherwise}\\
\end{array}\right.
$$
Monotone family of increasing sequences: One can define a $C^*$-homomorphism $\pi$ from $M_i(n,k)$ to $B(\bigotimes\limits_{i=1}^k \hh_i)$
$$
\pi(u_{i,j})=\left\{\begin{array}{lcl}
 \bigotimes\limits_{m_1=1}^{i-1} I_{l_{m_1}} \otimes P_{e^{(l_i)}_{j'}}\bigotimes\limits_{m_2=i+1}^{k} P_{l_{m_2}} & \text{if} &  0<j'=j-\sum\limits_{l=m}^{i-1}l_m\leq l_i\\
0& \text{if} &  \text{otherwise}\\
\end{array}\right.
$$
The existence of these homomorphisms are given by the universal conditions for $A_i(n,k)$, $B_i(n,k)$ and $M_i(n,k)$ respectively. Since the above representation of  $M_i(n,k)$ plays an important role in proving our main theorems, we summarize it as the following proposition:
\begin{proposition}\label{nontrivial representation}  For fixed natural numbers $n>k$. Let $l_1,...,l_k\in\mathbb{N}$ such that $l_1+\cdots+l_k=n.$ Let $\hh_{i}$ be a $l_i$-dimensional  Hilbert spaces with orthonormal basis $\{e^{(i)}_{j}|j=1,...,l_i\}$ and $I_{l_i}$ be the unit of  the algebra $B(\hh_{l_i})$, $P_{e^{(l_i)}_j}$ be the one dimensional orthogonal projection onto $\C e^{(l_i)}_j$,  $P_i$ be the one dimensional projection onto $\C\sum\limits_{j}e^{(l_i)}_j$. Then, there is a $C^*$-homomorphism $\pi:M_i(n,k)\rightarrow B(\hh_1\otimes\cdots\otimes\hh_k)$ defined as follows:
$$
\pi(u_{i,j})=\left\{\begin{array}{lcl}
 \bigotimes\limits_{m_1=1}^{i-1} I_{l_{m_1}} \otimes P_{e^{(l_i)}_{j'}}\bigotimes\limits_{m_2=i+1}^{k} P_{l_{m_2}} & \text{if} &  0<j'=j-\sum\limits_{l=m}^{i-1}l_m\leq l_i\\
0& \text{if} &  \text{otherwise}\\
\end{array}\right.
$$
\end{proposition}
Also, we need the following property in the future:

\begin{lemma}\label{extension}  Given natural numbers $n_1, n_2, n,k\in\mathbb{N}$ such that $n>k$. Let $(u_{i,j})_{i=1,...,n;j=1,..., k}$ be the standard generators of $M_i(n,k)$ and  $(u'_{i,j})_{i=1,...,n+n_1+n_2;j=1,..., k+n_1+n_2}$ be the standard generators of $M_i(n+n_1+n_2,k+n_1+n_2)$. Then, there exists a $C^*$-homomorphism $\pi:M_i(n+n_1+n_2,k+n_1+n_2)\rightarrow M_i(n,k)$ such that 
 $$
\pi(u'_{i,j})=\left\{\begin{array}{ccc}
\delta_{i,j}\p &  \text{if}   & 1\leq i\leq n_1 \\
\delta_{i,j}\ &  \text{if}   & n_1+1\leq i\leq  n+n_1,\,\,  n_1 \leq j\leq n_1+k \\
0 &  \text{if}   & n_1+1\leq i\leq  n+n_1,\,\,  j\leq n_1\,\, \text{or}\,\, j>n_1+k  \\
\delta_{i-n,j-k}I &  \text{if}   & i\geq n+n_1+1 \\
\end{array}\right.
$$
where $\p=P_1=\sum\limits_{i=1}^nu_{i,1}$ and $I$ is the identity of $M_i(n,k)$.
\end{lemma}
\begin{proof}
We can see that the matrix form of $(\pi(u'_{i,j}))_{i=1,...,n+n_1+n_2;j=1,..., k+n_1+n_2}$ is 
$$
     \left(\begin{array}{cccccccccc}
     \p &\cdots   &0 & 0&\cdots &0 &0 &\cdots &0 \\
       \vdots   &\ddots  &\vdots &\vdots &\ddots &\vdots &\vdots &\ddots &\vdots \\
0 &\cdots   &\p & 0&\cdots &0 &0 &\cdots &0 \\
     0&\cdots   &0 & u_{1,1}&\cdots &u_{1,k} &0 &\cdots &0 \\
     \vdots   &\ddots  &\vdots &\vdots &\ddots &\vdots &\vdots &\ddots &\vdots \\
     0&\cdots   &0 & u_{n,1}&\cdots &u_{n,k} &0 &\cdots &0 \\
     0 &\cdots   &0 & 0&\cdots &0 &I &\cdots &0 \\
     \vdots   &\ddots  &\vdots &\vdots &\ddots &\vdots &\vdots &\ddots &\vdots \\
     0 &\cdots   &0 & 0&\cdots &0 &0 &\cdots &I \\    
      \end{array}\right)
$$
It is easy to check that the coordinates of the above matrix satisfy the universal conditions of $M_i(n+n_1+n_2,k+n_1+n_2)$. The proof is complete.
\end{proof}
In analogue of  the $(n,k)$-partial exchangeability, we can define  noncommutative versions of partial exchangeability for free independence and boolean independence:

\begin{definition}\normalfont
For $k,n\in\mathbb{N}$ with $k\leq n$, the quantum space $A_l(n,k)$ is the universal unital $C^*$-algebra generated by elements $\{u_{ij}|1\leq i\leq n, 1\leq j\leq k\}$ such that
\begin{itemize}
\item[1.] Each  $u_{ij}$ is an orthogonal projection:$u_{ij}=u_{ij}^*=u_{ij}^2$.
\item[2.] Each column of the rectangular matrix $u=(u_{ij})$ forms a partition of unity: for $1\leq j\leq k$ we have $\sum\limits_{i=1}^nu_{ij}=1$.
\end{itemize}
\end{definition}

\begin{remark} $A_i(n,k)$ is a quotient algebra of $A_l(n,k)$, because the definition of $A_i(n,k)$ has one more restriction than $A_l(n,k)$'s. $A_l(n,n)$ is exactly Wang's quantum permutation group  $A_s(n)$.
\end{remark}
There is a well defined unital algebraic  homomorphism
 $$\alpha_{n,k}^{(fp)}:\C\langle X_1,...X_k\rangle\rightarrow \C\langle X_1,...X_n\rangle\otimes A_l(n,k)$$
 such that
 $$\alpha_{n,k}^{(fp)} X_j=\sum\limits_{i=1}^n X_i\otimes u_{i,j}$$
where $1\leq j\leq k$.
The distributional symmetry associated with this quantum structure is:
\begin{definition}\normalfont
Let $x_1,...,x_n\in (\A,\phi)$ be a sequence of $n$-noncommutative random variables, $k\leq n$ be a positive integer. We say the sequence is  $(n,k)$-quantum exchangeable if
$$\mu_{x_1,...x_k}(p)=\mu_{x_1,...x_n}\otimes id_{A_l(n,k)}(\alpha^{(fp)}_{n,k}(p)),$$
for all $p\in\mathbb{C}\langle X_1,...,X_k\rangle$, where $\mu_{x_1,...,x_j}$ is the joint distribution of $x_1,...x_j$  with respect to $\phi$ for $j=k,n$. 
\end{definition}

By modifying the second universal condition of $A_l(n,k)$, we can define a boolean version of partial exchangeability:

\begin{definition} For natural numbers $k\leq n$, $B_l(n,k)$ is the non-unital  universal $C^*$-algebra generated by the elements $\{u_{i,j}\}_{i=1,...,n;j=1,...,k}$ and an orthogonal projection $\p$, such that
\begin{enumerate}
\item $u_{i,j}$ is an orthogonal projection, i.e. $u_{i,j}=u_{i,j}^*=u_{i,j}^2$.
\item $\sum\limits_{i=1}^n u_{i,j}\p=\p$ for all $1\leq j\leq k$.
\end{enumerate}
\end{definition}
\begin{remark} $B_l(n,n)$ is exactly the boolean exchangeable quantum semigroup  $B_s(n)$.
\end{remark}

There is a well defined unital algebraic  homomorphism
 $$\alpha_{n,k}^{(bp)}:\C\langle X_1,...X_k\rangle\rightarrow \C\langle X_1,...X_n\rangle\otimes B_l(n,k)$$
 such that
 $$\alpha_{n,k}^{(bp)} X_j=\sum\limits_{i=1}^n X_i\otimes u_{i,j}$$
where $1\leq j\leq k$.
The distributional symmetry associated with this quantum structure is:
\begin{definition}\normalfont
Let $x_1,...,x_n\in (\A,\phi)$ be a sequence of $n$-noncommutative random variables, $k\leq n$ be a positive integer. We say the sequence is  $(n,k)$-boolean exchangeable if
$$\mu_{x_1,...x_k}(p)\p=\p\mu_{x_1,...x_n}\otimes id_{B_l(n,k)}(\alpha^{(bp)}_{n,k}(p))\p$$
for all $p\in\mathbb{C}\langle X_1,...,X_k\rangle$, where $\mu_{x_1,...,x_j}$ is the joint distribution of $x_1,...x_j$  with respect to $\phi$. 
\end{definition}

Now, we turn to define our noncommutative distributional symmetries for infinite sequences of random variables.  In this paper, our infinite ordered index set $I$ would be either $\mathbb{N}$ or $\mathbb{Z}$.

\begin{definition}\normalfont
Let $(\A,\phi)$ be a noncommutative probability space,  $I$ be  an ordered index set and $(x_i)_{i\in I}$ a sequence of random variables in $\A$.   $(x_i)_{i\in I}$ is said to be monotonically (boolean) spreadable if all its finite subsequences $(x_{i_1},...,x_{i_l})$ are monotonically(boolean) spreadable.
\end{definition}

\begin{lemma}\label{submono}
Let $(x_1,...,x_{n+1})$ be a monotonically spreadable sequence of random variables in $(\A,\phi)$. Then, all its subsequences are monotonically spreadable.
\end{lemma}
\begin{proof}
It suffices to show that  the subsequence $(x_1,...,x_{l-1},x_{l+1},...,x_{n+1})$ is monotonically spreadable for all $1\leq l \leq n$. If we denote $(x_1,...,x_{l-1},x_{l+1},...,x_{n+1})$ by $(y_1,...,y_{n})$, then we need to show that  $(y_1,...,y_{n})$ is $M_i(n,k)$-spreadable for all $ k<n $.\\
Fix $k<n$,  let $\{u_{i,j}\}_{i=1,...,n ;j=1,...,k}$ be the set of generators of $M_i(n,k)$  and $\{P_{i,j}\}_{i=1,...,n+1 ;j=1,...,k+1}$ be an $n+1$ by $k+1$ matrix with entries in $M_i(n,k)$ such that 
$$
P_{i,j}=\left\{\begin{array}{cc}
\um_{i,j}& \text{if}\,\,   1\leq i,j<l\\
\um_{i-1,j}& \text{if}\,\,   1\leq j<l ,\,  i\geq l\\
\um_{i,j-1}& \text{if}\,\,   1\leq i<l, \, j\geq l\\
\um_{i-1, j-1}& \text{if}\,\,   i,j \geq l\\
0& \text{otherwise}\,\,   \\
\end{array}\right..
$$

It is a routine to check that the set $\{P_{i,j}\}_{i=1,...,n+1 ;j=1,...,k+1}$ satisfies the universal conditions of $M_i(n+1,k+1)$. Therefore, there exists a $C^*$-homomorphism  $\psi:M_i(n+1,k+1)\rightarrow M_i(n,k)$ such that 
$$\psi(u'_{i,j})=P_{i,j}$$
where $\{u'_{i,j}\}$ is the set of generators of $M_i(n+1,k+1)$.  Now, we need a convenient notation:
$$\sigma(i)=
\left\{\begin{array}{cc}
i & \text{if}\,\,   1\leq i<l\\
i+1& \text{if}\,\, i\geq l\\
\end{array}\right.$$
Then, $P_{\sigma(i),\sigma(j)}=\um_{i,j}$ and $y_i=x_{\sigma(i)}$ for all $i=1,...,n$ and $j=1,....,k+1$. For all monomial $X_{j_1}\cdots X_{j_m}\in\C\langle X_1,...,X_k\rangle$, let $P_1'=\sum\limits_{i=1}^n u'_{i,1}$ and $\p$ be the invariance projection of $M_i(n,k)$, we have 
$$
\begin{array}{rcl}
&&\mu_{y_1,...,y_n}(X_{j_1}\cdots X_{j_m})\p \\
&=&\p \mu_{x_1,...,x_{n+1}}(X_{\sigma(j_1)}\cdots X_{\sigma(j_m)})\psi(P'_1)\p \\
&=& \p \psi(\mu_{x_1,...,x_{n+1}}(X_{\sigma(j_1)}\cdots X_{\sigma(j_m)})P'_1)\p \\
&=& \p \psi(\mu_{x_1,...,x_{n+1}}\otimes id_{M_i(n+1,k+1)}(\sum\limits_{i_1,...,i_m=1}^{n+1}X_{i_1}\cdots X_{i_m}\otimes u'_{i_1,\sigma(j_1)}\cdots u'_{i_m,\sigma(j_m)}))\p \\
\end{array}
$$
Notice that $u'_{l,\sigma(j)}=0$ since $\sigma(j)$ never equals $l$, the quality can be written as the following:

$$
\begin{array}{rcl}
&&\mu_{y_1,...,y_n}(X_{j_1}\cdots X_{j_m})\p \\
&=&\p  \psi(\mu_{x_1,...,x_{n+1}}\otimes id_{M_i(n+1,k+1)}(\sum\limits_{i_1,...,i_m=1}^{n}X_{\sigma(i_1)}\cdots X_{\sigma(i_m)}\otimes u'_{\sigma(i_1),\sigma(j_1)}\cdots u'_{\sigma(i_m),\sigma(j_m)}))\p \\
&=&\p \sum\limits_{i_1,...,i_m=1}^{n}\mu_{x_1,...,x_{n+1}}(X_{\sigma(i_1)}\cdots X_{\sigma(i_m)})\psi( u'_{\sigma(i_1),\sigma(j_1)}\cdots u'_{\sigma(i_m),\sigma(j_m)})\p \\
&=& \sum\limits_{i_1,...,i_m=1}^{n}\mu_{y_1,...,y_n}(X_{i_1}\cdots X_{i_m})\p\um_{i_1,j_1}\cdots \um_{i_m,j_m}\p \\
\end{array}
$$

which completes the proof.
\end{proof}

Then,  we have

\begin{proposition}\label{sub spreadable}  Let $(\A,\phi)$ be a noncommutative probability space and  $(x_i)_{i\in \mathbb{Z}}$ be a sequence of random variables in $\A$. Then, $(x_i)_{i\in \mathbb{Z}}$ is  monotonically (quantum, boolean) spreadable  if and only if  $(x_i)_{i=-n,-n+1,...,.n-1,n}$ is monotonically (quantum, boolean) spreadable for all $n$. 
\end{proposition}
\begin{proof}
It is sufficient to prove \lq\lq $\Leftarrow$\rq\rq. Given a subsequence $(x_{i_1},...,x_{i_l})$ of  $(x_i)_{i\in \mathbb{Z}}$,  there exits an $n$ such that $-n<i_1,...,i_l<n$.  Since $(x_i)_{i=-n,-n+1,...,.n-1,n}$ is monotonically  spreadable, by Lemma \ref{submono}, we have that  $(x_{i_1},...,x_{i_l})$ is monotonically spreadable. The same to quantum spreadability and boolean spreadability.
\end{proof}

\section{ Relations between noncommutative probabilistic symmetries}
In this section, we will study relations between the  noncommutative distributional symmetries which are introduced in the previous section.

It is well know that every $C^*$-algebra admits a faithful representation.  Fix $n,k\in\mathbb{N}$, such that $1\leq k\leq n-1$.  Let $\Phi$ be a faithful representation of $ B_l(n,k)$ into $B(\hh)$ for some Hilbert space $\hh$. For convenience, we denote $\Phi(u_{i,j})$ by $u_{i,j}$ and $\Phi(\p)$ by $\p$. 

According to the definition of $B_l(k,n)$, $u_{i,j}$'s and $\p$ are orthogonal projections in $B(\hh)$. Let $Q_i=\sum\limits_{j=1}^k u_{i,j}$  for $1\leq i\leq n$.  In \cite{KR}, we know that the set $P(\hh)$ of orthogonal projections on $\hh$ is a lattice  with respect to the usual order $\leq$ on the set of selfadjoint operators, i.e. two selfadjoint operators $A$ and $B$, $A\leq B$ iff $B-A$ is a positive operator.   

Now, we need the following notation in our construction.  Given two projections $E$ and $F$,  we denote by $E\vee F$ the minimal orthogonal projection in $P(\hh)$, such that $E\vee F$ is greater or equal to $E$ and $F$. $E\vee F$ is well define and unique,  we call it the supreme of $E$ and $F$. It is easy to see that $(E\vee F) E=E$ and $(E\vee F) F=F$

We turn to define a sequence of orthogonal projections $\{P'_i\}_{i=1,...,n}$ in $P(\hh) $ as follows:
$$ P'_1=I-Q_1,$$
$$ P'_i=I-P'_1\vee\cdots\vee P'_{i-1}\vee Q_i $$
for $2\leq i\leq n$.

\begin{lemma}\label{sup} Given a nonzero  vector $v\in\hh$,  $E$ and $F$ are two orthogonal projections on $\hh$. If  $(E\vee F)x=x$ and $ Ex=0$, then $Fx=x$.
\end{lemma} 

According the construction of $\{P'_i\}_{1\leq i\leq n}$,  we have 
$$P'_iP'_j=\delta_{i,j}P'_i$$ and 
$$P_i'u_{i,j}=0$$
for all $1\leq i\leq n$ and $1\leq j\leq k$.

\begin{lemma}\label{adding} $\sum\limits_{i=1}^nP_i'=I$, where $I$  is the identity in $B(\hh)$.
\end{lemma}
\begin{proof}
Since the orthogonal projections $P'_i$ are orthogonal to each other, $\sum\limits_{i=1}^n P_i'$ is an orthogonal projection which is less than or equal to the identity $I$. If $\sum\limits_{i=1}^n P_i'< I$, then there exists a nonzero vector $v\in\hh$ such that $$ \sum\limits_{i=1}^n P_i' v=0.$$
Then, we have $$0=P_i'x=(I-P'_1\vee\cdots\vee P'_{i-1}\vee P_i )x$$
or say
$$ (P'_1\vee\cdots\vee P'_{i-1}\vee P_i )x=x$$
for all $i$.
 Since $P'_mx=0$ for all $1\leq m\leq i-1$, by Lemma\ref{sup}, $P_ix=x$.  Then, we have 
$$
\begin{array}{rcl}
nx&=&\sum\limits_{i=1}^n P_ix\\
&=&\sum\limits_{i=1}^n\sum\limits_{j=1}^k u_{i,j}x\\
&=&\sum\limits_{j=1}^k(\sum\limits_{i=1}^n u_{i,j}x)\\
\end{array},
$$
which implies that $n$ is in the spectrum of $ \sum\limits_{j=1}^k\sum\limits_{i=1}^n u_{i,j}$.
Notice that, to every $1\leq j\leq k$, $\sum\limits_{i=1}^n u_{i,j}\leq I$ since they are orthogonal projections and orthogonal to each other. Therefore,
$$0\leq \sum\limits_{j=1}^k\sum\limits_{i=1}^n u_{i,j}\leq \sum\limits_{j=1}^k I\leq kI.$$
It contradicts to the implication above. The proof is complete.
\end{proof}

\begin{corollary} $\sum\limits_{i=1}^nP_i'\p=\p.$
\end{corollary}

\begin{proposition}
Let $(\A,\phi)$ be a noncommutative probability space,  $(x_i)_{i=1,...,n}$ is a finite ordered sequence of random variables in $\A$. For fixed $n>k$,  the  joint distribution $\mu_{x_1,...,x_n}$ is $A_l(n,k)$-invariant if it is $A_l(n,k+1)$-invariant
\end{proposition}
\begin{proof}
Let $\{u_{ij}|1\leq i\leq n, 1\leq j\leq k\}$ the set of  standard generators of $A_l(n,k)$, $\Phi$ be a faithful representation of $A_l(n,k)$ into $B(\hh)$. With the above construction, we can define  $\{u'_{i,j}\}_{i=1,...,n;j=1,...,k+1}$ as following:
$$
u_{i,j}'=\left\{\begin{array}{cc}
\Phi(u_{i,j}) & \text{if}\,\,\,\,\, j\leq k\\
P_j' & \text{if}\,\,\,\,\, j=k\\
\end{array}\right.
$$
By Lemma \ref{adding}, $\{u'_{i,j}\}_{i=1,...,n;j=1,...,k+1}$ satisfies the universal conditions for $A_l(n,k+1)$. Let $\{u''_{ij}|1\leq i\leq n, 1\leq j\leq k+1\}$ be the set of  standard generators of $A_l(n,k+1)$. then there exists a $C^*-$homomorphism $\Phi:A_l(n,k+1)\rightarrow B(\hh)$ such that:
$$\Phi'(u''_{ij})=u'_{i,j}.$$
Therefore, $\Phi^{-1}\Phi'$ defines a unital $C^*-$homomorphism
$$\Phi^{-1}\Phi': C^*-alg\{u'_{i,j}|1\leq i\leq n, 1\leq j\leq k\}\rightarrow A_l(n,k)$$
such that 
$$\Phi^{-1}\Phi'(u'_{i,j})=u_{i,j}$$
for all $1\leq i\leq n, 1\leq j\leq k$.

If $\mu_{x_1,...,x_n}$ is $A_l(n,k+1)$-invariant, then 
$$\mu_{x_1,...,x_{k+1}}(p)=\mu_{x_1,...,x_{k}}\otimes id_{A_l(n,k+1)}(\alpha^{(fp)}_{n,k+1}(p))$$
for all $p\in\C\langle X_1,...,X_{k+1}\rangle$.
Let $p=X_{j_1}\cdots X_{j_l}\in\C\langle X_1,...,X_{k}\rangle$, then we have 
$$ 
\begin{array}{rcl}
&&\mu_{x_1,...,x_{k}}(P)1_{A(n,k)}\\
&&\Phi^{-1}\Phi'(\mu_{x_1,...,x_{k+1}}(P)1_{A(n,k+1)}))\\
&=&\Phi^{-1}\Phi'(\mu_{x_1,...,x_{n}}\otimes id_{A_l(n,k+1)}(\alpha^{(fp)}_{n,k+1}(X_{j_1}\cdots X_{j_l}))\\
&=&\Phi^{-1}\Phi'(\mu_{x_1,...,x_{n}}\otimes id_{A_l(n,k+1)}(\sum\limits_{i_1,...,i_l}^n X_{i_1}\cdots X_{i_l}\otimes u'_{i_1,j_1}\cdots u'_{i_l,j_l})\\
&=&\mu_{x_1,...,x_{n}}\otimes id_{A_l(n,k)}(\sum\limits_{i_1,...,i_l}^n X_{i_1}\cdots X_{i_l}\otimes u_{i_1,j_1}\cdots u_{i_l,j_l})\\
&=&\mu_{x_1,...,x_{n}}\otimes id_{A_l(n,k)}(\alpha^{(fp)}_{n,k}(P))
\end{array}
$$
Since $p$ is an arbitrary monomial, the proof is complete.
\end{proof}

The same, we can show that
\begin{corollary}
$\mu_{x_1,...,x_n}$ is $B_l(n,k)$-invariant if it is $B_l(n,k+1)$-invariant
\end{corollary}

\begin{lemma}
$\mu_{x_1,...,x_n}$ is $(n,k)$-quantum spreadable if it is $A_l(n,k)$-invariant.
\end{lemma}
\begin{proof}  
Let $\{u_{i,j}\}_{i=1,...,n;j=1,...,k}$ be generators of  $A_i(n,k)$ and $\{u'_{i,j}\}_{i=1,...,n;j=1,...,k}$ be generators of  $A_l(n,k)$. Then, there is a well defined $C^*$-homomorphism $\beta:A_l(n,k)\rightarrow A_i(n,k)$ such that $\beta(u_{i,j}=u'_{i,j})$. The existence of $\beta$ is given by the universality of $A_l(n,k)$. Since $\mu_{x_1,...,x_n}$ is $A_l(n,k)$-invariant, for all monomials $p=X_{i_1}\cdots X_{i_m}\in\C\langle X_1,...,X_k\rangle$, we have 
$$\mu_{x_1,...,x_k}(p)1_{A_{l}(n,k)}= \mu_{x_1,...x_n}\otimes id_{A_l(n,k)}(\alpha^{(fp)}_{n,k}(p))=\sum\limits_{j_1,...,j_m} \phi(x_{j_1}\cdots x_{j_m})u_{j_1,i_1}\cdots u_{j_m,i_m}.$$
Apply $\beta$ on both sides of the above equation, we have 
$$\mu_{x_1,...,x_k}(p)1_{A_{i}(n,k)}=\sum\limits_{j_1,...,j_m} \phi(x'_{j_1}\cdots x_{j_m})u'_{j_1,i_1}\cdots u_{j_m,i_m}=\mu_{x_1,...x_n}\otimes id_{A_i(n,k)}(\alpha_{n,k}(p)).$$
The proof is complete.
\end{proof}
The same, we have 
\begin{corollary}
$\mu_{x_1,...,x_n}$ is $(n,k)$-boolean spreadable if it is $B_l(n,k)$-invariant.
\end{corollary}

\begin{corollary}
$(x_1,...,x_n)$ is  boolean spreadable if it is boolean exchangeable. $(x_1,...,x_n)$ is quantum spreadable if it is quantum exchangeable.
\end{corollary}

In summary, for fixed $n,k\in\mathbb{N}$ such that $k<n$, we have the following diagrams:

\begin{displaymath}
    \xymatrix{ B(n,n)_\text{inv}\ar[r]\ar[dd]& B_l(n,k)_\text{inv} \ar[r]\ar[dd] & B_i(n,k)_\text{inv} \ar[d] \\  
                         &    & M_i(n,k)_\text{inv} \ar[d] \\ 
                A(n,n)_\text{inv}\ar[r] & A_l(n,k)_\text{inv} \ar[r] & A_i(n,k)_\text{inv}  \\ }
\end{displaymath}

and

\begin{displaymath}
    \xymatrix{ Booolean\,\, exchangeability \ar[r]\ar[dd]&  Boolean\,\, spreadability \ar[d] \\  
                         &   Monotone\,\, spreadability\ar[d] \\ 
                Quantum\,\, exchangeability \ar[r] & quantum\,\,spreadability  \\ }
\end{displaymath}
The arrow \lq\lq$\text{condtion a)}\rightarrow \text{condition b)}$" means that condition a) implies condition b).

\section{Monotonically equivalent sequences }
In order to study monotone spreadability, we need find  relations between mixed moments of monotonically spreadable sequences of random variables. Since all the mixed moments can be denoted by  finite sequences of indices, we will turn to study finite sequences of ordered indices. In this section, we introduce an equivalent relation, which has a deep relation with monotone spreadability, on finite sequences of ordered indices.   
\begin{definition}\normalfont
Given two pairs of integers $(a, b)$, $(c, d)$ , we say these two pairs have the same order if  $a-b, c-d$ are both  positive or negative or 0. 
\end{definition}
For example, $(1, 2)$ and $(3, 5)$ have the same order but $(1, 2)$ and $(5, 3)$ do not have the same order.

\begin{definition}\normalfont  Let $\mathbb{Z}$ be the set of integers with natural order and  $\mathbb{Z}^L=\mathbb{Z}\times \cdots \times \mathbb{Z}$ be the set of finite sequences of length $L$. We define a partial relation $\sim_m$ on $\mathbb{Z}^L$.  Given two sequences of indices $\I=\{i_1,....,i_L\}, \J=\{j_1,...,j_L\}\in \mathbb{Z}^L$. If for all $1\leq l_1<l_2\leq L$ such that $i_{l_3}> \max\{i_{l_1}, i_{l_2}\}$ for all $l_1<l_3<l_2$,  $(i_{l_1},i_{l_2})$ and $(j_{l_1},j_{l_2})$ have the same order, then we denote $\I\sim_m\J$. 
\end{definition}
{\bf Example:}   $(5,3,4)\sim_m(5,3,5)$ but $(5,6,4)\not\sim_m(5,6,5)$. 

\begin{remark}
In general, the relation can be defined on any ordered set but not only $\mathbb{Z}$. We will show this partial relation is exactly an equivalence relation on the set of finite sequences of ordered indices.
\end{remark}
It follows the definition that $(i_{l},i_{l+1})$ and $(j_{l},j_{l+1})$  have the same order for all $1\leq l<L$   if $\I\sim_m \J$.\\
Now we turn to show that $\sim_m$ is actually an equivalent relation.  To achieve it,  we need to show  that the relation $\sim_m$ is  reflexive, symmetric and transitive.\\
(Reflexivity) First, reflexivity  is obvious, because a pair $(i_{l_1}, i_{l_2})$  always has the same order with itself. 
\begin{lemma}\label{symmetry}
(Symmetry) Let  $\I=\{i_1,....,i_L\}, \J=\{j_1,...,j_L\}\in \mathbb{Z}^L$ such that $\I\sim_m\J$, then $\J\sim_m\I$. 
\end{lemma}
\begin{proof}
Suppose that $\J\not\sim_m\I$. 
Then,  there exist two natural numbers  $1\leq l_1<l_2\leq L$ such that $$j_{l_3}> \max\{j_{l_1}, j_{l_2}\}$$ for all $l_1<l_3<l_2$, but  $(j_{l_1},j_{l_2})$ and $(i_{l_1},i_{l_2})$ do not have the same order.  Fix $l_1$,  we choose the  smallest $l_2$ which satisfies the above property.  Notice that $\I\sim_m\J$, $(j_{l_1},j_{l_1+1})$ and $(i_{l_1},i_{l_1+1})$ have the same order, then 
$$l_2\neq l_1+1.$$
  According to our assumption,  we have $$j_{l'_3}> \max\{j_{l_1},j_{l_2}\}$$
  for  $l_1<l'_3<l_2$. \\
Suppose that there exists an $l_3$  between $l_1$ and $l_2$ such that $$i_{l_3}\leq \max\{i_{l_1},i_{l_2}\}.$$ 
Without loss of generality, we assume that $$i_{l_1}\geq i_{l_2},$$then $$i_{l_3}\leq i_{l_1}.$$ Again, among these $l_3$, we choose the smallest one. Then, we have  $i_l> i_{l_1}\geq i_{l_3}$ for $$l_1<l<l_3.$$
Since  $\I\sim_m\J$, $(i_{l_1},i_{l_3})$ and $(j_{l_1},j_{l_3})$ must have the same order, but $i_{l_1}\geq i_{l_3}$ and $i_{l_1}<j_{l_3}$. It contradicts the existence of our $l_3$. Hence, $i_{l'_3}>\max\{i_{l_1},i_{l_2}\}$ for all $l_1<l'_3<l_2$.  It follows that   $(i_{l_1},i_{l_2})$ and $(j_{l_1},j_{l_2})$ have the same order. But, it contradicts our original assumption.  Therefore, $\J\sim_m\I$.\\
\end{proof}

\begin{lemma} Given two sequences  $\I=\{i_1,....,i_L\}, \J=\{j_1,...,j_L\}\in \mathbb{Z}^L$ such that $I\sim_m\J$. Let $1\leq l_1<l_2\leq L$ such that $i_{l_3}>\max{i_{l_1},i_{l_2}}$ for all $l_1<l_3<l_2$. Then, we have $$j_{l_3}>\max\{j_{l_1},j_{l_2}\}$$ for all $l_1<l_3<l_2$.
\end{lemma}\label{connected}
\begin{proof}
If the statement is false, then there exists $l_3$ between $l_1$ and $l_2$ such that $$j_{l_3}\leq \max \{j_{l_1},j_{l_2}\}.$$ Suppose $j_{l_1}\geq j_{l_2}$, then $$j_{l_3}\leq j_{l_1}.$$  Among all these $l_3$, we take the smallest one. Then, we have $$j_{l_4}>\max\{j_{l_1},j_{l_3}\}$$ for all  $ l_1<l_4<l_3$.  By Lemma\ref{symmetry}, $\J\sim_m \I$ since $\I\sim_{m}\J$.  Therefore, $(j_{l_1},j_{l_3})$ and $(i_{l_1},i_{l_3})$ must have the same order which means  $$i_{l_{1}}\geq i_{l_3}.$$ 
This is a contradiction. If we assume that $j_{l_1}<j_{l_2}$, then we just need  to consider the largest one among those $l_3$ and  we will get the same contradiction. The proof is complete.
\end{proof}

\begin{lemma}\label{Trasitivity}
(Transitivity)Given  three sequences $\I=\{i_1,....,i_L\}, \J=\{j_1,...,j_L\}, \Q=\{q_1,...q_L\}\in \mathbb{Z}^L$, such that $\I\sim_m\J$ and $\J\sim_m\Q$, then $\I\sim_m\J$
\end{lemma}
\begin{proof}
Given  $1\leq l_1<l_2\leq L$ such that $$i_{l_3}>\max\{i_{l_1},i_{l_2}\}$$ for all $l_1<l_3<l_2$. By Lemma \ref{connected},  we have
$$j_{l_3}>\max\{j_{l_1},j_{l_2}\}$$ for all $l_1<l_3<l_2$. 
It follows the definition that $(i_{l_1},i_{l_2})$, $(j_{l_1},j_{l_2})$ have the same order and $(j_{l_1},j_{l_2})$, $(q_{l_1},q_{l_2})$ have the same order. Therefore, $(i_{l_1},i_{l_2})$, $(q_{l_1},q_{l_2})$ have the same order. Since $l_1,l_2$ are arbitrary, it completes the proof.
\end{proof}
By now, we have shown that the relation $\sim_m$ is reflexive, symmetric and transitive. Therefore, we have
\begin{proposition}
$\sim_m$ is an equivalence relation on $\mathbb{Z}^L$.
\end{proposition}

As we mentioned before,  $\mathbb{Z}$ can be replaced by any ordered set $I$. When there is no confusion, we always use  $\sim_m$ to denote the monotone equivalence relation on $I^L$ for ordered set $I$ and positive integers $L$. For example, $I$ can be $[n]=\{1,...,n\}$.

\begin{definition}\normalfont Let $\I=(i_1,...,i_L)$ be a sequence of ordered indices.  An ordered subsequence  $(i_{l'_1},...,i_{l'_2})$ of $\I$ is called an interval if the sequence contains all the elements $i_{l'_3}$ whose position  $l'_3$ is between $l'_1$ and $l'_2$. An interval  $(i_{l'_1},...,i_{l'_2})$ of  $\I$  is called a crest if $i_{l_1'}=i_{l_1'+1}\cdots=i_{l'_2}>\max\{ i_{l_1'-1},i_{l_2'+1}\}$. In addition , we assume that  $i_0<i_1$ and $i_L>i_{L+1}$ even though $i_0, i_{L+1}$ are not in $\I$. 
\end{definition}
{\bf Example:}   $(1,2,3,4)$  has one crest of length 1, namely $(4)$. $(1,2,1,3,4,4,3,5)$ has 3 crests $(2),(4,4),(5)$ and $2$ is the first peak of the sequence. $(1,1,1,1,1)$ has one crest $(1,1,1,1,1)$ which is the sequence itself, because we assumed $i_0<i_1$ and $i_6<i_5$.

\begin{lemma}
Given $\I=(i_1,...,i_L)\in\mathbb{Z}^L$, $\I$ has at least one crest.
\end{lemma}
\begin{proof}
Since $\I$ consists of  finite elements, it has a maximal one, i.e. $i_l$ such that $i_l\geq i_{l'}$ for $1\leq l'\leq L$ . It is obvious that $i_l$ must be contained in an interval $(i_{l'_1},...,i_{l_2'}) $ such that $$i_{l_1'}=i_{l_1'+1}\cdots=i_{l'_2}=i_{l'}$$ and $$i_{l'}>\max\{ i_{l_1'-1},i_{l_2'+1}\}.$$  Therefore, $\I$ contains a crest.
\end{proof}

\begin{lemma}\label{remove }
Given two index sequences $\I,\J\in\mathbb{Z}^L$ such that $\I\sim_m\J$.  If $(i_{l'_1},...,i_{l'_2})$ is a crest of $\I$, then $(j_{l'_1},...,j_{l'_2})$ is a crest of $\J$
\end{lemma}
\begin{proof}
Since $\I\sim_m\J$,  all consecutive  pairs $(i_l,i_{l+1})$ and $(j_l,j_{l+1})$  have the order.  According to the definition, we have 
$$i_{l'_1-1}<i_{l_1'}=i_{l_1'+1}\cdots=i_{l'_2}>j_{l'_2+1} $$
If follows that 
$$j_{l'_1-1}<j_{l_1'}=j_{l_1'+1}\cdots=j_{l'_2}>j_{l'_2+1} ,$$
thus $(j_{l'_1},...,j_{l'_2})$ is a crest of $\J$.
 
\end{proof}

Now, we will introduce some $\sim_m$ preserving operations on index sequences. The first operation is to remove a crest from a sequence.  Let $(i_{l'_1},...,i_{l'_2})$ be an interval of $\I=(i_1,...,i_L)$, we denote by $\I\setminus(i_{l'_1},...,i_{l'_2})$ the new sequence $(i_1,...,i_{l'_1-1},i_{l'_2+1},...,i_L) $.  We denote by the empty set $\emptyset=\I\setminus \I$ and  we assume $\emptyset\sim_m\emptyset$.

\begin{lemma}\label{remove1 }
Let $\I=(i_1,....,i_l), \J=(j_1,...,j_L)\in \mathbb{Z}^L$ such that $\I\not\sim_m\J$. If $(i_{l'_1},...,i_{l'_2})$ is a crest of $\I$ and $(j_{l'_1},...,j_{l'_2})$ is a crest of $\J$.  Then, 
$$\I\setminus(i_{l'_1},...,i_{l'_2})\not\sim_m \J\setminus (j_{l'_1},...,j_{l'_2})$$
\end{lemma}
\begin{proof}
If $\I\setminus(i_{l'_1},...,i_{l'_2})$ is empty,  then $\J\setminus (j_{l'_1},...,j_{l'_2})$ must be empty because the lengths of $\I$, $\J$ are the same.
If $\I\setminus(i_{l'_1},...,i_{l'_2})$ is non empty, then $\I$ can be written as 
$$  (i_1,...,i_{l'_1},...,i_{l'_2},...,i_{L}) $$
and  
$$\I\setminus(i_{l'_1},...,i_{l'_2})=(i_1,...,i_{l'_1-1},i_{l'_2+1},...,i_{L})= (i'_1,...,i'_{l'_1-1},i'_{l'_1},...,i'_{L-l'_2+l'_1-1})$$ and 
$$\J\setminus(j_{l'_1},...,j_{l'_2})=(j_1,...,j_{l'_1-1},j_{l'_2+1},...,j_{L})= (j'_1,...,j'_{l'_1-1},j'_{l'_1},...,j'_{L-l'_2+l'_1-1})$$
For any indices $1\leq l_1<l_2<L-l'_2+l'_1-1$ such that $i_{l_3}>\max\{i'_{l_1},i'_{l_2}\}$:\\
If $ l_1,l_2\leq l'_1-1$ or $ l_1,l_2\geq l'_1$, then $(i'_{l_1},...,i'_{l_2})$ is an interval of $\I$.  Since $\I\sim_m\J$,  $(i'_{l_1},i'_{l_2})$ and $(j'_{l_1},j'_{l_2})$ have the same order .\\
If $ l_1< l'_1\leq l_2$, then $i'_{l_2}=i_{l_2+l'_2-l'_1+1}$.  We have 
$$i_{l_3}>i_{l'_1-1}\geq \max\{i'_{l_1},i'_{l_2} \}$$
for all $l'_1\leq l_3\leq l'_2$.  It follows that 
$$i_{l_3}>\max\{i_{l_1},i_{l_2} \}$$
for all $l_1<l_3<l_2+l'_2-l'_1+1$. It follows that $(i_{l_1},i_{l_2+l'_2-l'_1+1})$ and $(j_{l_1},j_{l_2+l'_2-l'_1+1})$ have the same order.  Thus, 
$(i'_{l_1},i'_{l_2})$ and $(j'_{l_1},j'_{l_2})$ have the same order. \\
The proof is complete.
\end{proof}

The same as the previous proof, by checking the definition of $\sim_m$, we have 
\begin{lemma}\label{increasing}
Let $\I=(i_1,....,i_L)\in \mathbb{Z}^L$ and $(i_{l'_1},...,i_{l'_2})$  is a crest of $\I$,  then  we have 
$$\I=(i_1,...i_L)\sim_m (i_1,...,i_{l'_1-1},i_{l'_1}+K,...,i_{l'_2}+K , i_{l'_2+1},...,i_l)$$
for any integer $K$ such that $i_{l'_1}+K>\max\{i_{l'_1-1},i_{l'_2+1}\}$.
\end{lemma}

The following proposition shows a deep relation between the set of standard generators of $\M(n,k)$ and $\sim_m$: 
\begin{proposition}\label{orthogonal}
Given two sequences $\I=\{i_1,...,i_L\}\in [k]^L,\J=\{j_1,...,j_L\}\in [n]^L$, let $\{\um_{i,j}\}_{i=1,...,n; j=1,..., k}$ be the set of standard generators of $\M(n,k)$,  then  we have 
$$\sum\limits_{(q_1,...,q_L)\sim_m \J}\um_{q_1,i_1}\cdots \um_{q_L,i_L}\p=\left\{\begin{array}{cc}
\p &  \text{ if}\,\,\,\, \J\sim_m\I\\
0 & \text{otherwise}
\end{array}\right.
$$
\end{proposition}
\begin{proof}
We will prove the proposition by induction. \\
When $L=1$,   the statement is apparently true.\\
Suppose the statement is true for all $L\leq L'$. Let us consider the case  $L=L'+1$. Let $(i_{l'_1},...,i_{l'_2})$ be a crest of $\I$:\\
{\bf Case 1:} If $(j_{l'_1},...,j_{l'_2})$ is not a crest of $\J$, then $\I\not\sim_m \J$  and  one of the following cases  happens:
\begin{itemize}
\item[1.] There exists an index $j_{l'_3}$  of $\J$ such that $j_{l'_3}\neq j_{l'_3+1}$ for some $ l'_1\leq l'_3<l'_2$. 
\item[2.]$j_{l'_1}\leq j_{l'_1-1}$. 
\item[3.] $j_{l'_2}\leq j_{l'_2+1}$. 
\end{itemize}

But, for all $\Q=(q_1,...,q_L)\sim_m\J$, we have:
\begin{itemize}
\item[1.]  $(q_{l'_3},q_{l'_3-1}) $ and $(j_{l'_3},j_{l'_3-1}) $ have the same order.
\item[2.]  $(q_{l'_1},q_{l'_1-1}) $ and $(j_{l'_1},j_{l'_1-1}) $ have the same order. 
\item[3.]  $(q_{l'_2},q_{l'_2+1}) $ and $(j_{l'_2},j_{l'_2+1}) $ have the same order. 
\end{itemize}
Therefore, we have:
\begin{itemize}
\item[1.]  $q_{l'_3}\neq q_{l'_3-1} $ and $i_{l'_3}=i_{l'_3-1} $ for some $ l'_1\leq l'_3<l'_2$.
\item[2.]  $q_{l'_1}\leq q_{l'_1-1} $ and $i_{l'_1}>i_{l'_1-1}$.
\item[3.]  $q_{l'_2}\leq q_{l'_2+1} $ and $i_{l'_2}>i_{l'_2+1} $. 
\end{itemize}
 According to the definition of $M_i(n,k)$,  we have one of the following equations:
 \begin{itemize}
\item[1.]  $\um_{q_{l'_3},i_{l'_3}} \um_{q_{l'_3+1},i_{l'_3+1}}=0$ for some $ l'_1\leq l'_3<l'_2$.
\item[2.]  $\um_{q_{l'_1-1},i_{l'_1-1}} \um_{q_{l'_1},i_{l'_1}}=0$.
\item[3.]  $\um_{q_{l'_2},i_{l'_2}} \um_{q_{l'_2+1},i_{l'_2+1}}=0$.
\end{itemize}
In this case,  we have  
$$
\sum\limits_{(q_1,...,q_L)\sim_m \J}\um_{q_1,i_1}\cdots \um_{q_L,i_L}\p=0.
$$
{\bf Case 2:} If $(j_{l'_1},...,j_{l'_2})$ is  a crest of $\J$, then $(q_{l'_1},...,q_{l'_2})$ is  a crest of $\Q$. Therefore,
$$\um_{q_{l'_1},i_{l'_1}}\cdots \um_{q_{l'_2},i_{l'_2}}=\um_{q_{l'_1},i_{l'_1}}.$$

By Lemma \ref{increasing}, if we fix the indices of $\Q\setminus (q_{l'_1},...,q_{l'_2})$, then $q_{l'_1},...,q_{l'_2}$  can be any integers such that $q_{l'_1}=...=q_{l'_2}$ and $\max\{q_{l'_1-1},q_{l'_2+1})\}<q_{l'_1}\leq n$.  Therefore, we have   
$$
\begin{array}{rcl}
&&\sum\limits_{\max\{q_{l'_1-1},q_{l'_2+1})\}<q_{l'_1}\leq n}\um_{q_{l'_1-1},i_{l'_1-1}}\um_{q_{l'_1},i_{l'_1}}\um_{q_{l'_2+1},i_{l'_2+1}}\\
&=&\sum\limits_{1\leq q_{l'_1}\leq n}\um_{q_{l'_1-1},i_{l'_1-1}}\um_{q_{l'_1},i_{l'_1}}\um_{q_{l'_2+1},i_{l'_2+1}}\\
&=&\um_{q_{l'_1-1},i_{l'_1-1}}\um_{q_{l'_2+1},i_{l'_2+1}}\\
\end{array}.
$$
The first equality holds because the extra terms are 0. The second equality uses the monotone universal condition of $M_i(n,k)$. Let $L''=L-l'_2+l'_1+1\leq L'$, then $\J\setminus (j_{l'_1},...,j_{l'_2})\in[n]^{L''}$
By Lemma \ref{remove }, $ \Q\setminus (q_{l'_1},...,q_{l'_2})\sim_m \J\setminus (j_{l'_1},...,j_{l'_2})$. If we denote by $(i'_1,...,i'_{L''})$ the sequence $\I\setminus (i_{l'_1},...,i_{l'_2})$, then we have  
\[
\begin{array}{rcl}
&&\sum\limits_{(q_1,...,q_L)\sim_m \J}\um_{q_1,i_1}\cdots \um_{q_L,i_L}\p\\
&=&\sum\limits_{(q'_1,...,q'_{L''})\sim_m \J\setminus (j_{l'_1},...,j_{l'_2})}\um_{q'_1,i'_1}\cdots \um_{q'_{L''},i'_{L''}}\p\\
&=&\left\{\begin{array}{cc}
\p &\text{ if}\,\,\,\, \J\setminus (j_{l'_1},...,j_{l'_2})\sim_m  \I\setminus (i_{l'_1},...,i_{l'_2})\\
0 & \text{otherwise}
\end{array}\right.
\end{array}
\]

The last equality comes from the assumption of our induction. By Lemma\ref{remove } and Lemma\ref{remove1 }, $\J\setminus (j_{l'_1},...,j_{l'_2})\sim_m  \I\setminus (i_{l'_1},...,i_{l'_2})$ iff $\J \sim_m \I$.\\
 The proof is complete.
\end{proof}

\subsection{Operator valued monotone sequences are monotonically spreadable}
In this subsection, we will show that operator valued monotone finite sequences of random variables are monotonically spreadable. To achieve it,  we need to consider the positions of the smallest elements of  indices sequences.

\begin{definition}\normalfont Let $\I=(i_1,...,i_L)$ be a sequence of ordered indices and $a=\min\{i_1,...,i_L\}$. We call the set $\S(\I)=\{l|i_l=a\}$ the positions of the smallest elements of $\I$. An interval of $(i_{l'_1},...,i_{l'_2})$ is called a hill of $\I$ if $ i_{l'_1-1}=i_{l'_2+1}=a$ and $i'_{l_3}\neq a$ for all $l'_1\leq l'_3\leq l'_2$, here we assume $i_0=i_{L+1}=a$ for convenience.
\end{definition}
{\bf Example:}  $(1,2,3,4,1,2,1)$ has two hills $(2,3,4)$ and $(2)$.  $(1,2,1,3,4,)$ has two hills $(2)$ and $(3,4)$. $(1,1,1,1,1)$ has no hill.
 
\begin{lemma}\label{minimal} Given two sequences $\I=\{i_1,...,i_L\}, \J=\{j_1,...,j_L\}\in [n]^L$ such that $\I\sim_m\J $, then $\S(\I)=\S(\J)$. Let $(i_{l'_1},...,i_{l'_2})$ be a hill of $\I$, then $$(i_{l'_1},...,i_{l'_2})\sim_m(j_{l'_1},...,j_{l'_2}).$$
\end{lemma}
\begin{proof}
Let us check the values of $\J$ one by one. Suppose $$\S(\I)=\{l''_1<\cdots<l''_{k'}\},$$ where $k'$ is the number of elements of $\S(\I)$. Let $b=\min\{j_1,...j_L\}$,  we want to show that $j_{l''_1}=\cdots=j_{l''_{k'}}=b$ and $j_{l}>b$ for all $l\not\in\S(\I)$.\\
Given an integer  $1\leq p<k'$, we have   $$i_l>a=i_{l''_p}=i_{l''_{p+1}}$$ for all $l''_p<l<l''_{p+1}$.
According to  the definition of $\sim_m$ and Lemma\ref{connected},  we have  $$j_{l''_p}=j_{l''_{p+1}}$$  and $$j_l> \max\{j_{l''_p},j_{l''_{p+1}}\}$$ for all $l''_p<l<l"_{p+1}$. The left is to check the elements $j_l$ with $l<l''_1$ or $l>l''_{k'}. $ If there exists and $l<l''_1$ such that $j_l\leq j_{l''_1}$, we chose the greatest such $l$. Then,  we have $$j_{l'}>\max\{j_l,j_{l''_1}\}$$ for all $l<l'<l''_1$. Therefore, we have  $$i_l\leq i_{l''_1}$$ which is a contradiction.  It implies that 
$$j_{l}>j_{l''_1}$$ for all $l<l''_1$.
the same we have 
$$j_{l}>j_{l''_1}$$ for all $l>l''_k$.
Therefore, $j_{l''_1}=\cdots=j_{l''_{k'}}=\min\{j_1,...,j_L\}$. The last statement is obvious from the definition of $\sim_m$. The proof is complete.
\end{proof}

Given $\I=\{i_1,...,i_L\}\in\mathbb{Z}^L$, we will denote by $x_{\I}=x_{i_1}x_{i_2}\cdots x_{j_L}$ for short .  Then, we have 

\begin{proposition}
Let $(\A,\B, E)$ be an operator valued probability space, and $(x_i)_{i=1,...,n}$ be a sequence of random variables in $\A$. If $(x_i)_{i=1,...,n}$ are identically distributed and monotonically independent. Then,  for indices sequences $\I=\{i_1,...,i_L\},\J=\{j_1,...,j_L\}\in [n]^L$ such that $\I\sim_m\J$, $L\in\mathbb{N}$,  we have 
$$E[x_{\I}]=E[x_{\J}].$$
\end{proposition}
\begin{proof}
When $L=1$, the statement is true since the sequence is identically distributed.\\
Suppose the statement is true for all $L\leq L'\geq 1$.
Let us consider the case  $L=L'+1$. If $\I$ has no hill, then $i_1=\cdots=i_L$ which implies $j_1=\cdots=j_L$. 
The statement is true for this case, because the sequence is identically distributed. 
Suppose $\I$ has hills $\I_1,...,\I_l$ and $a=\min\{i_1,...,i_L\}$. 
Then, $x_{\I}$ can be written as 
$$x_a^{n_1}x_{\I_1}x_a^{n_2}x_{\I_2}\cdots x_a^{n_l}x_{\I_l}x_a^{n_{l+1}},$$
where $n_2,...,n_l\in\mathbb{Z}^+$ and $n_1,n_{l+1}\in\mathbb{Z}\cup\{0\}$. 
Since $(x_i)_{i=1,...,n}$ are monotonically independent, we have
$$E[x_\I]=E[x_a^{n_1}E[x_{\I_1}]x_a^{n_2}E[x_{\I_2}]\cdots x_a^{n_l}E[x_{\I_l}]x_a^{n_{l+1}}].$$
Let $b=\min\{j_1,...,j_L\}$, by Lemma \ref{minimal}, $\J$ has hills $\J_1,...,\J_l$ whose positions of elements correspond to the positions of elements of $\I_1,...,\I_l$ and $\J_{l'}\sim_m\J_{l'}$ for all $1\leq l'\leq k'$. Therefore, we have 
$$\begin{array}{rcl}
E[x_\J]&=&E[x_b^{n_1}E[x_{\J_1}]x_b^{n_2}E[x_{\J_2}]\cdots x_b^{n_l}E[x_{\J_l}]x_b^{n_{l+1}}]\\
&=&E[x_b^{n_1}E[x_{\I_1}]x_b^{n_2}E[x_{\I_2}]\cdots x_b^{n_l}E[x_{\I_l}]x_b^{n_{l+1}}]\\
&=&E[x_a^{n_1}E[x_{\I_1}]x_a^{n_2}E[x_{\I_2}]\cdots x_a^{n_l}E[x_{\I_l}]x_a^{n_{l+1}}]\\
&=&E[x_\I],
\end{array}
$$
where the second equality follows the induction and the third equality holds because $x_a$ and $x_b$ are identically distributed. The proof is complete.
\end{proof}

\begin{proposition}\label{mi is ms}
Let $(\A,\B, E)$ be an operator valued probability space, and $(x_i)_{i=1,...,n}$ be a sequence of random variables in $\A$. If $(x_i)_{i=1,...,n}$ are identically distributed and monotonically independent with respect to $E$.  Let $\phi$ be a state on $\A$ such that $\phi(\cdot)=\phi(E[\cdot])$. Then, $(x_i)_{i=1,...,n}$ is monotonically spreadable with respect to $\phi$.
\end{proposition}
\begin{proof}
For fixed natural numbers $n,k\in\mathbb{N}$, let $(u_{i,j})_{i=1,...,n; j=1,...,k}$ be standard generators of $M_i(n,k)$. Let $\J=(j_1,...,j_L)\in [k]^L$ and denote $x_{j_1}\cdots x_{j_L}$ by $x_{\J}$.   We denote the equivalent class of $[n]^L$ associated with $\sim_m$ by $\overline{[n^L]}$. For each $\I\in [n]^L$, we denote $u_{i_1,j_1}\cdots u_{i_L,j_L}$ by $u_{\I,\J}$. Then, by proposition \ref{orthogonal}, we have 
$$\begin{array}{rcl}
&&\sum\limits_{\I\in [n]^L}\phi(x_{\I}) \p u_{\I,\J} \p\\
&=&\sum\limits_{\I\in [n]^L}\phi(E[x_{\I}])\p u_{\I,\J}\p\\
&=&\sum\limits_{\bar Q\in \overline{[n]^L}}\sum\limits_{\I\in \bar Q}\phi(E[x_{\I}]) \p u_{\I,\J}\p\\
&=&\sum\limits_{\J\not\in\bar Q\in \overline{[n]^L}}\sum\limits_{\I\in \bar Q}\phi(E[x_{\I}]) \p u_{\I,\J}\p+\sum\limits_{\J\in\bar Q\in \overline{[n]^L}}\sum\limits_{\I\in \bar Q}\phi(E[x_{\I}]) \p u_{\I,\J}\p\\
&=&\sum\limits_{\J\not\in\bar Q\in \overline{[n]^L}}\sum\limits_{\I\in \bar Q}\phi(E[x_{\Q}]) \p u_{\I,\J}\p+\sum\limits_{\I\sim_m\J}\phi(E[x_{\J}]) \p u_{\I,\J}\p\\
&=&0+\phi(E[x_{\J}])\p\\
&=&\phi(x_{\J})\p\\
\end{array}
$$
Since $n,k$ are arbitrary, the proof is complete.

\end{proof}

\section{Tail algebras}
In the previous work on distributional symmetries, infinite sequences of objects are indexed by natural numbers.  
For this kind of  infinite sequences of random variables,  the conditional expectations in de Finetti type theorems are defined via the limit of  unilateral shifts. 
It is shown in \cite{Ko}  that unilateral shift is an isometry  frome $\A$ to itself if  $(\A,\phi)$ is a $W^*$-probability space generated by a spreadable sequence of random variables and  $\phi$ is faithful. 
Therefore,   WOT continuous conditional expectations defined via the limit of unilateral shift exist in a very weak situation, i.e. the sequence of random variables just need to be spreadable.  However,  our works are in a more general situation that the state $\phi$ is not necessarily faithful. 
 In our framework, we will provide an example in which the sequence of random variables  is  monotone spreadable but the unilateral shift is not an isometry. Therefore,  we can not get an extended de Finetti type theorem for monotone independence in the usual way. The key change in this paper is that we will  consider bilateral sequences of random variables.   We begin with an interesting example :
\subsection{Unbounded spreadable sequences}  Unlike the situation in probability spaces with faithful states, an infinite spreadable sequence of random variables indexed by natural numbers needs not to be bounded. Even more, there exists  an infinite monotonically spreadable unbounded sequence of bounded random variables in a non-degenerated $W^*$-probability space.\\
\noindent{\bf Example:}
Let  $\hh$ be the standard 2-dimensional Hilbert space with orthonormal basis $$\{v=\left(
\begin{array}{c}
1\\
0\\
\end{array}
\right),
w=\left(
\begin{array}{c}
0\\
1\\
\end{array}
\right)
\}.$$ Let $ p, A, x \in B(\hh)$  be operators on $\hh$ with the following matrix forms:
$$p=\left(
\begin{array}{cc}
1&0\\
0&0\\
\end{array}
\right),\,\,\,\,\,A=\left(
\begin{array}{cc}
1&0\\
0&2\\
\end{array}
\right),\,\,\,\,\,x=\left(
\begin{array}{cc}
0&1\\
1&0\\
\end{array}
\right).
$$

Let $\mathscr{H}=\bigotimes\limits_{n=1}^{\infty}\hh$ the infinite tensor product of $\hh$.  Let $\{x_i\}_{i=1}^\infty$ be a sequence of selfadjoint operators in $B(\mathscr{H})$ defined as follows:
$$x_i=\bigotimes\limits_{n=1}^{i-1}A\otimes x\otimes \bigotimes\limits_{m=1}^{\infty}p$$

Let $\phi$ be the vector state $\langle \cdot v,v\rangle$ on $\hh$ and $\Phi=\bigotimes\limits_{n=1}^{\infty}\phi$ be a state on $B(\mathscr{H})$. It is obvious that   $\Phi(x_i^n)=\phi(x^n)$ for for $i$. Therefore, the sequence $(x_i)_{i\in\mathbb{N}}$ is identically distributed.  For any $x,y \in B(\hh)$, an elementary computation shows $$\phi(xpy)=\phi(x)\phi(y).$$
For convenience, we  will denote  $A^{\otimes i-1}=\bigotimes\limits_{n=1}^{i-1}A$ and $P^{\otimes \infty}=\bigotimes\limits_{n=1}^{\infty}P$. Also, we  denote $x_{i_1}\cdots x_{i_L}=x_{\I}$ for $\I=(i_1,...,i_L)\in\mathbb{N}^L$ . We will show that the sequence $\{x_i\}_{i\in\mathbb{N}}$ is $M_i(n, k)$-spreadable with respect to $\Phi$.
\begin{lemma}
For indices sequences $\I=(i_1,...,i_L),\J=(j_1,...,j_L)\in [n]^L$ such that $\I\sim_m\J$ and $L\in\mathbb{Z}^+$,  we have 
$$\Phi(x_{\I})=\Phi(x_{\J})$$
\end{lemma}

\begin{proof}
When $L=1$, the statement is true since the sequence is identically distributed.\\
Suppose the statement is true for all $L\leq L'$. Let us consider the case  $L=L'+1$. If $\I$ has no hill, then $i_1=\cdots=i_L$ which implies $j_1=\cdots=j_L$. The statement is true for this case, because the sequence is identically distributed.  Also, we denote by $x_i^{(n)}$ the $n$-the component of $x_i$. Then,  
$$
x_i^{(n)}=\left\{
\begin{array}{cc}
a&\text{if}\,\,\, n<i\\
 x&\text{if}\,\,\, n=i\\
p&\text{if}\,\,\, n>i\\
\end{array}
\right.
$$
and $x^{(n)}_{\I}=x^{(n)}_{i_1}x^{(n)}_{i_2}\cdots x^{(n)}_{i_L}$.\\
 According to the definition of $\Phi$, we have that 
 $$ \Phi(x_{i_1}x_{i_2}\cdots x_{j_L})=\prod\limits_{n=1}^{\infty}\phi(\prod\limits_{l=1}^Lx_i^{(n)}).$$
 Notice that all the terms $\phi(\prod\limits_{l=1}^Lx_i^{(n)})$ are $1$ except finite terms.
Suppose $\I$ has hills $\I_1,...,\I_l$ and $a=\min\{i_1,...,i_L\}$, then $x_{\I}$ can be written as 
$$x_a^{n_1}x_{\I_1}x_a^{n_2}x_{\I_2}\cdots x_a^{n_l}x_{\I_l}x_a^{n_{l+1}}.$$
Therefore,
$$
\phi(\prod\limits_{l=1}^Lx_i^{(n)})=\left\{
\begin{array}{cc}
1&\text{if}\,\,\, n<a\\
\phi(x^{n_1}A^{|\I_1|}x^{n_2}A^{|\I_2|}\cdots x^{n_l}A^{|\I_l|}x^{n_{l+1}})&\text{if}\,\,\, n=a\\
\phi(px^{(n)}_{\I_1}px^{(n)}_{\I_2}p\cdots px^{(n)}_{\I_l}p)&\text{if}\,\,\, n>a\\
\end{array}
\right.
$$
It follows that
$$\phi(\prod\limits_{l=1}^Lx_i^{(n)})=\prod\limits_{n\geq \min\{\I\}}^{\infty}\phi(\prod\limits_{l=1}^Lx_i^{(n)}).
$$
Because 
$$\phi(px^{(n)}_{\I_1}px^{(n)}_{\I_2}p\cdots px^{(n)}_{\I_l}p)=\phi(x^{(n)}_{\I_1})\phi(x^{(n)}_{\I_2})\cdots \phi(x^{(n)}_{\I_l}),$$
we have 
$$\begin{array}{rcl}
&&\Phi(x_{i_1}x_{i_2}\cdots x_{j_L})\\
&=& \phi(x^{n_1}A^{|\I_1|}x^{n_2}A^{|\I_2|}\cdots x^{n_l}A^{|\I_l|}x^{n_{l+1}}) \prod\limits_{n>a}^{\infty}\phi(px^{(n)}_{\I_1}px^{(n)}_{\I_2}p\cdots px^{(n)}_{\I_l}p).\\
&=&\phi(x^{n_1}A^{|\I_1|}x^{n_2}A^{|\I_2|}\cdots x^{n_l}A^{|\I_l|}x^{n_{l+1}}) \prod\limits_{n>a}^{\infty}\phi(x^{(n)}_{\I_1})\phi(x^{(n)}_{\I_2})\cdots \phi(x^{(n)}_{\I_l})\\
&=&\phi(x^{n_1}A^{|\I_1|}x^{n_2}A^{|\I_2|}\cdots x^{n_l}A^{|\I_l|}x^{n_{l+1}}) \Phi(x_{\I_1})\Phi(x_{\I_2})\cdots \Phi(x_{\I_l})\\
\end{array}
$$
Let $b=\min\{j_1,...,j_L\}$, by Lemma\ref{minimal}, $\J$ has hills $\J_1,...,\J_l$ whose positions of elements correspond to the positions of elements of $\I_1,...,\I_l$ and $\J_{l'}\sim_m\J_{l'}$ for all $1\leq l'\leq k'$. 
Therefore, we have 
$$\begin{array}{rcl}
\Phi(x_{\J})&=&\Phi(x_{i_1}x_{i_2}\cdots x_{i_L})\\
&=&\phi(x^{n_1}A^{|\J_1|}x^{n_2}A^{|\J_2|}\cdots x^{n_l}A^{|\J_l|}x^{n_{l+1}}) \Phi(x_{\J_1})\Phi(x_{\J_2})\cdots \Phi(x_{\J_l})\\
&=&\phi(x^{n_1}A^{|\I_1|}x^{n_2}A^{|\I_2|}\cdots x^{n_l}A^{|\I_l|}x^{n_{l+1}}) \Phi(x_{\I_1})\Phi(x_{\I_2})\cdots \Phi(x_{\I_l})\\
&=&\Phi(x_{\I})
\end{array}
$$
where the second equality follows the induction and the true that $\J_k\sim_m\I_k$ and $|\J_k|=|\I_k|$ for all $1\leq k\leq l$.   The proof is complete.
\end{proof}

\begin{proposition}
The joint distribution of $(x_i)_{i\in\mathbb{N}}$  with respect to $\Phi$  is monotonically spreadable.
\end{proposition}
\begin{proof}
Fixed $n>k\in\mathbb{N}$,  let $\{\um_{i,j}\}_{i=1,...,n; j=1,..., k}$ be the set of standard generators of $\M(n,k)$. For all $\I=(i_1,...,i_L)\in[k]^L$,  we denote by $\overline{[n]^L}$ the $\sim_m$ equivalence class of $[n]^L $, then
we have 
$$
\begin{array}{rcl}
&&\p\mu_{x_1,...,x_n}\otimes(id_{\M(n,k)})(\am_{n,k}(X_{\I}))\p\\
&=&\sum\limits_{\J\in[n]^L}\mu_{x_1,...,x_n}(X_{\J})\p\um_{\J,\I}\p\\
&=&\sum\limits_{\bar \Q \in \overline{[n]^L}}\sum\limits_{\J\in \bar{\Q}}\mu_{x_1,...,x_n}(X_{\J})\p\um_{\J,\I}\p\\
&=&\sum\limits_{\I\not\in\bar \Q\in \overline{[n]^L}}\sum\limits_{\J\in \bar{\Q}}\mu_{x_1,...,x_n}(X_{\J})\p\um_{\J,\I}\p+\sum\limits_{\J\sim_m\I}\mu_{x_1,...,x_n}(X_{\J})\p\um_{\J,\I}\p\\
&=&\sum\limits_{\I\not\in\bar \Q\in \overline{[n]^L}}\sum\limits_{\J\in \bar{\Q}}\mu_{x_1,...,x_n}(X_{\Q})\p\um_{\J,\I}\p+\sum\limits_{\J\sim_m\I}\mu_{x_1,...,x_n}(X_{\I})\p\um_{\J,\I}\p\\
&=&\sum\limits_{\I\not\in\bar \Q\in \overline{[n]^L}}\mu_{x_1,...,x_n}(X_{\Q})\sum\limits_{\J\in \bar{\Q}}\p\um_{\J,\I}\p+\sum\limits_{\J\sim_m\I}\mu_{x_1,...,x_n}(X_{\I})\p\um_{\J,\I}\p\\
&=&\sum\limits_{\I\not\in\bar \Q\in \overline{[n]^L}}\mu_{x_1,...,x_n}(X_{\Q})\cdot 0+\sum\limits_{\J\sim_m\I}\mu_{x_1,...,x_n}(X_{\I})\p\um_{\J,\I}\p\\
&=&\sum\limits_{\J\sim_m\I}\mu_{x_1,...,x_n}(X_{\I})\p\um_{\J,\I}\p\\
&=&\Phi(x_{\I})\p
\end{array}
$$
 
The proof is complete. 
 
\end{proof}
By direct computations, we have 
$$ \prod\limits_{i=1}^n x_{n+1-i} v^{\otimes\infty}= w^{\otimes n}\otimes v^{\otimes \infty}$$
and 
\begin{equation} \label{eq:norm}
x_{n+1} w^{\otimes n}\otimes v^{\otimes \infty}=2^n w^{\otimes n+1}\otimes v^{\otimes \infty}
\end{equation}
Let $(\hh',\pi',\xi') $ be the GNS representation of the von Neumann algebra generated by $(x_i)_{i=1,...,\infty}$ associated with $\Phi$. We have 
$$\|\pi'(x_{n+1})\|\leq \|x_{n+1}\|=2^{n},$$
but equation \ref{eq:norm} shows that 
$\|\pi'(x_{n+1})\|\geq 2^n$. Therefore, $\|\pi'(x_{n+1})\|=2^n$.

Therefore, there is no bounded endomorphism $\alpha$ on  $\A$ such that $\alpha(x_i)=x_{i+1}$.

\subsection{Tail algebras of bilateral sequences of random variables}  
In the last subsection, we showed that, in a $W^*$-probability space with a non-degenerated normal state, the unilateral shift of a spreadable unilateral sequence of random variables may not be extended to be a bounded endomorphism. Therefore, in general,  we can not define a normal condition expectation by taking the limit of unilateral shifts of variables.  The main reason here is that  the spreadability of variables does not give enough restrictions to control the norms of the variables in our probability space.  In $(\A,\phi)$, a $W^*$-probability space with a faithful state, the norm of a selfadjoint random variable $x\in\A$ is controlled by the moments of $X$, i.e. 
$$\| x\|=\lim\limits_{n\rightarrow \infty}\phi(|x|^n)^{\frac{1}{n}}.$$
But, in our non-degenerated $W^*$-probability spaces, the norm of a random variable depends on all  mixed moments which involve it.  To make the conditional expectation exist, we will consider spreadable sequences of random variables indexed by $\mathbb{Z}$ but not $\mathbb{N}$. In this case, the sequence $(x_i)_{i\in\mathbb{Z}}$ is bilateral. As a consequence, we will have two choices to take limits on defining normal conditional expectations and tail algebras. Before studying properties of tail algebras of bilateral sequences, we introduce some necessary notations and assumptions first.

Let $(\A,\phi)$  is a $W^*-$probability space generated by a spreadable bilateral  sequence of bounded random variables $(x_i)_{i\in\mathbb{Z}}$ and $\phi$ is a non-degenerated normal state.  
We assume that the unit of $\A$ is contained in the WOT-closure of the non-unital algebra generated by $(x_i)_{i\in\mathbb{Z}}$.  
Let  $(\HH, \pi, \xi)$ be the GNS representation of $\A$ associated with $\phi$. Then,  $\{\pi(P(x_i|i\in\mathbb{Z}))\xi | P\in\C\langle X_i |   i\in\mathbb{Z}\rangle\}$ is dense in $\HH$.  For convenience, we will denote  $\pi(y)\xi$ by $\hat y$ for all $y\in\A$. When there is no confusion, we will write $y$ short for $\pi(y)$.  
We  denote by $A_{k+}$ the non-unital algebra generated by $(x_i)_{i\geq k}$ and $A_{k-}$ the non-unital algebra generated by $(x_i)_{i\leq k}$. Let $\A_k^+$ and $\A_k^-$ be the WOT-closure of  $A_{k+}$ and  $A_{k-}$, respectively.

\begin{definition}\normalfont
Let $(\A,\phi)$ be a no-degenerated noncommutative $W^*$-probability space, $(x_i)_{i\in\mathbb{Z}}$ be a bilateral sequence of bounded random variables in $\A$ such that $\A$ is the WOT closure of the non-unital algebra generated by $(x_i)_{i\in\mathbb{Z}}$. The positive tail algebra $\ATP$ of $(x_i)_{i\in\mathbb{Z}}$ is defined as following:
$$ \ATP=\bigcap\limits_{k>0} \A^+_k.$$
In the opposite direction, we define the negative tail algebra $\ATN$ of $(x_i)_{i\in\mathbb{Z}}$ as following:
$$ \ATN=\bigcap\limits_{k<0} \A^-_k.$$
\end{definition}

\begin{remark}
In general, the positive tail algebra and the negative tail algebra are different.
\end{remark}

Even though our framework looks quit different from the framework in \cite{Ko}, we can show that there exists a normal bounded shift of the sequence in a similar way. For completeness, we provide the details here.

\begin{lemma}
There exists a unitary map $U:\HH\rightarrow \HH$ such that $ U(P(x_i| i\in\mathbb{Z}))\xi=P(x_{i+1}| i\in\mathbb{Z})\xi$
\end{lemma}
\begin{proof}
Since $(x_i)_{i\in\mathbb{Z}}$ is spreadable,  we have
$$ \phi((P(x_i|i\in\mathbb{Z}))^*P(x_i|i\in\mathbb{Z}))= \phi((P(x_{i+1}|i\in\mathbb{Z}))^*P(x_{i+1} |i\in\mathbb{Z})).$$
It implies that 
 $$U( P(x_i|i\in\mathbb{Z})\xi)=P(x_{i+1}|i\in\mathbb{Z})\xi$$ 
 is a well defined  isometry on $\{\pi(P(x_i|i\in\mathbb{Z}))\xi | P\in\C\langle X_i |   i\in\mathbb{Z}\rangle\}$ .
 Since $\{\pi(P(x_i|i\in\mathbb{Z}))\xi | P\in\C\langle X_i |   i\in\mathbb{Z}\rangle\}$ is dense in $\HH$, $U$ can be extended to the whole space $\HH$.  It is obvious that $\{\pi(P(x_i|i\in\mathbb{Z}))\xi | P\in\C\langle X_i |   i\in\mathbb{Z}\rangle\}$ is contained  in the range of $U$.  Therefore, the extension of $U$ is a unitary map on $\HH$. 
\end{proof}

Now, we can define an automorphism $\alpha$ on  $\A$ by the following formula:
$$\alpha(y)=UyU^{-1}.$$

\begin{lemma}
$\alpha$ is the bilateral shift of $(x_i)_{i\in\mathbb{Z}}$, i.e. $$\alpha(x_k)=x_{k+1}$$ for all $k\in\mathbb{Z}$.
\end{lemma}
\begin{proof}
For all $y=P(x_i|i\in\mathbb{Z})\xi $, we have
$$\alpha(x_k)y=Ux_kU^{-1}P(x_i|i\in\mathbb{Z})\xi=Ux_kP(x_{i-1}|i\in\mathbb{Z})\xi=x_{k+1}P(x_i|i\in\mathbb{Z})\xi.$$
By the density of  $\{\pi(P(x_i|i\in\mathbb{Z}))\xi | P\in\C\langle X_i |   i\in\mathbb{Z}\rangle\}$,  we have $\alpha(x_k)=x_{k+1}$.
The proof is complete.
\end{proof}

Since $\alpha$ is a normal automorphism of $\A$, we have
\begin{corollary}
For all $k\in\mathbb{Z}$, we have $\alpha(\A_{k}^+)=\A_{k+1}^+$.
\end{corollary}

\begin{lemma}\label{6.7}
Fix $n\in\mathbb{Z}$. Let $y_1, y_2\in A_{n-}$. Then, we have 
$$ \langle \alpha^l(a)\hat y_1,\hat y_2 \rangle=\langle a\hat y_1,\hat y_2 \rangle,$$
where $l\in\mathbb{N}$ and $a\in\A_{n+1}^+$.
\end{lemma}
\begin{proof}
It is sufficient to prove the statement under the assumption that $l=1$. Since $a\in\A_{n+1}^+$, by Kaplansky's theorem, there exists a sequence $(a_{m})_{m\in\mathbb{N}}\subset A_{(n+1)+}$ such that $\|a_{m}\|\leq\|a\|$ for all $m$ and $a_{m}$ converges to $a$ in WOT. Then, by the  spreadability of $(x_i)_{i\in\mathbb{Z}}$, we have 
$$
\langle \alpha(a)\hat y_1,\hat y_2 \rangle=
\lim\limits_{m\rightarrow \infty}\langle \alpha(a_m)\hat y_1,\hat y_2 \rangle=\lim\limits_{m\rightarrow \infty}\phi(y_2^*a_m\hat y_1)=\langle a\hat y_1,\hat y_2 \rangle
$$
\end{proof}

In the following context, we fix $k\in\mathbb{Z}$.

\begin{lemma}\label{fixed point}
 For all $a\in\A_k ^+$, we have that
$$E^+[a]=WOT-\lim\limits_{l\rightarrow \infty} \alpha^l(a)$$ exists.   Moreover, $E^+[a]\in\ATP$
\end{lemma}
\begin{proof}
For all  $y_1, y_2\in\{\pi(P(x_i|i\in\mathbb{Z}))\xi | P\in\C\langle X_i |   i\in\mathbb{Z}\rangle\}$, there exits $n\in\mathbb{Z}$ such that $y_1,y_2\in\A_{n-}$. For all $l> n-k$,  we have $\alpha^l(a)\in\A_{(n+1)+}$. By Lemma \ref{6.7}, we have 
$$\langle \alpha^{n+1-k}(a)y_1,y_2\rangle=\langle \alpha^{n+2-k}(a)y_1,y_2\rangle=\cdots.$$
Therefore,
$$\lim\limits_{l\rightarrow \infty}\langle \alpha^l(a) y_1,y_2\rangle=\langle \alpha^{n+1-k}(a)y_1,y_2\rangle.$$
$ \alpha^l(a)$  converges pointwisely to an element $E^+[a]$.  Since for all $n>0$, we have $\alpha^{l}(a)\in \A_n^+$ for all $l>n-k+1$. It follows that 
$WOT-\lim\limits_{l\rightarrow \infty} \alpha^l(a)\in\A_n^+$ for all $n$. Hence, $E^+[a]\in\ATP$.
\end{proof}

\begin{proposition}
$E^+$ is normal on $\A^+_k$ for all $k\in\mathbb{Z}$.
\end{proposition}
\begin{proof}
Let $(a_m)_{m\in\mathbb{N}}\subset\A^+_k$ be a bounded sequence which converges to $0$ in WOT.  For all  $y_1, y_2\in\{\pi(P(x_i|i\in\mathbb{Z}))\xi | P\in\C\langle X_i |   i\in\mathbb{Z}\rangle\}$, there exits $n\in\mathbb{Z}$ such that $y_1,y_2\in\A_{n-}$. Then, we have
$$\lim\limits_{m\rightarrow \infty} \langle E^+[a_{m}]y_1,y_2\rangle=\lim\limits_{m\rightarrow \infty} \langle \alpha^{n+1-k}(a_m)y_1,y_2\rangle=0.$$
The last equality holds because $\alpha^l$ is normal for all $l\in\mathbb{N}$. The proof is complete.
\end{proof}

\begin{remark}
$E^+$ is defined on $\bigcup\limits_{k\in\mathbb{Z}} \A_{k}^+$ but not on $\A$. In general, we can not extend $E^+$ to the whole algebra $\A$.
\end{remark}

\begin{lemma} \label{6.11}$E^+[a]=a$ for all $a\in\ATP$.
\end{lemma}
\begin{proof}
 For all  $\hat y_1,\hat  y_2\in\{\pi(P(x_i|i\in\mathbb{Z}))\xi | P\in\C\langle X_i |   i\in\mathbb{Z}\rangle\}$, there exits $n\in\mathbb{Z}$ such that $y_1,y_2\in A_{n-}$. Since $a\in\ATP\subset \A_{n+1}^+$, by Kaplansky's theorem,  there exists a sequence of $(a_m)_{m\in\mathbb{N}}\subset A_{(n+1)+}$ such that $a_m\rightarrow a$ in WOT and $\|a_m\|\leq \|a\|$ for all $m$. Then we have.
 $$\langle a \hat y_1,\hat y_2\rangle=\lim\limits_{m\rightarrow \infty} \langle a \hat y_1,\hat y_2\rangle=\lim\limits_{m\rightarrow \infty} \langle \alpha(a_m) \hat y_1,\hat y_2\rangle= \langle \alpha(a)\hat  y_1, \hat y_2\rangle.$$ 
Since $y_1, y_2$ are arbitrary, we have $a=\alpha(a)$.
\end{proof}

\begin{remark}  One should be careful that $\ATP$ could be a proper subset of the fixed points set of $\alpha$.
\end{remark}

\begin{lemma} $$E^+[a_1ba_2]=a_1E^+[b]a_2$$ for all $ b\in\A^+_k$, $a_1, a_2 \in\ATP$.
\end{lemma}
\begin{proof}
By Lemma\ref{6.11}, we have 
$$E^+[a_1ba_2]=\lim\limits_{l\rightarrow\infty} \alpha^l(a_1ba_2)= \lim\limits_{l\rightarrow\infty} \alpha^l(a_1)\alpha^l(b)\alpha^l(a_2)= \lim\limits_{l\rightarrow\infty} a_1\alpha^l(b)a_2=a_1E^+[b]a_2$$
\end{proof}

\section{Conditional expectations of bilateral monotonically spreadable sequence}

In this section, we assume that the joint distribution of  $(x_i)_{i\in\mathbb{Z}}$ is monotonically spreadable.
 
\begin{lemma}\label{7.1}
Fix $n>k\in\mathbb{N}$, let $(u_{i,j})_{i=1,...,n;\, j=1,...,k}$ be the standard generators of $M_i(n,k)$. Then, we have 
$$\phi(a_1x^{l_1}_{i_1}b_1x^{l_2}_{i_2}b_2\cdots b_{m-1}x^{l_m}_{i_m}a_2)\p=\sum\limits_{j_1,...,j_m=1}^n\phi(a_1x^{l_1}_{j_1}b_1x^{l_2}_{j_2}b_2\cdots b_{m-1}x^{l_m}_{j_m}a_2)\p u_{j_1,i_1}\cdots u_{j_m,i_m}\p,$$
where $1\leq i_1,...i_m\leq k$,  $b_1,...,b_{m-1}\in A_{(n+1)+}$ and $a_1,a_2\in A_{0-}$.
\end{lemma}
\begin{proof}
Without loss of generality, we assume that there exist $n_1,n_2\in\mathbb{N}$ such that $$a_1,a_2\in A_{[-n_1+1,0]}$$ and $$b_1,...,b_{m-1}\in A_{[n+1,n_2+k]}.$$  Since the map is linear, we just need to consider the case that $a_1, a_2$ and $b_1,...,b_{m-1}$ are products of $(x_i)_{i\in\mathbb{Z}}$.  
Let $$a_1=x_{s_{1,1}}\cdots x_{s_{1,t_1}}$$ and $$a_2=x_{s_{2,1}}\cdots x_{s_{2,t_2}}$$ for some $t_1,t_2\in\mathbb{N}$  and $-n_1+1\leq s_{c,d}\leq0$. Let 
$$b_i=x_{r_{i,1}}\cdots x_{r_{i,t'_{i}}}$$ for $t'_1,...,t'_{m-1}\in\mathbb{N}\cup\{0\}$ and $n+1\leq r_{c,d}\leq k+n_2$. Then,  $(x_{-n_1+1},...,x_{n+n_2})$ is a sequence of length $n+n_1+n_2$, we denote it by $(y_1,...,y_{n+n_1+n_2})$. Let $n'=n+n_1+n_2$ and $k'=k+n_1+n_2$. 
By our assumption, $a_1x^{l_1}_{i_1}b_1x^{l_2}_{i_2}b_2\cdots b_{m-1}x^{l_m}_{i_m}a_2$ is in the algebra generated by $(y_1,....,y_{k'})$. Let  $(u'_{i,j})_{i=1,...,n';\, j=1,...,k'}$ be the standard generators of $M_i(n',k')$ and $\p'$ be the invariant projection. 
 Let $\pi$ be the $C^*$-homomorphism in Lemma \ref{extension} and $id$ be the identity may on $\C\langle X_1,....,X_{n'}\rangle$.
 Since  $1\leq s_{c,d}+n_1\leq n_1$,  we have 
$$id\otimes\pi(\alpha_{n',k'}^{(m)}(X_{s_{i,1}+n_1}\cdots X_{s_{i,t_1}}+n_1))=X_{s_{i,1}+n_1}\cdots X_{s_{i,t_1}+n_1}\otimes\p.$$ 

Since  $n_1+n+1\leq r_{c,d}+n_1\leq n_1+n_2+k$, we have 
$$id\otimes\pi(\alpha_{n',k'}^{(m)}(X_{r_{i,1}+n_1}\cdots X_{r_{1,t'_1}}+n_1))=X_{r_{i,1}+n_1+n-k}\cdots X_{r_{i,t'_i}+n_1+n-k}\otimes I,$$ 
where $I$ is the identity of $M_i(n,k)$.
According to our assumption,  we have  $1\leq i_t \leq k$ for $t=1,....,m$. Then
$$id\otimes\pi(\alpha_{n',k'}^{(m)}(X^{l_t}_{i_t+n_1})=\sum\limits_{j_t=1}^{n} X^{l_t}_{j_t+n_1}\otimes u_{j_t,i_t}.$$

According to the monotone spreadability of $(y_1,...,y_{n'})$ and Lemma\ref{extension}, we have 
$$
\begin{array}{rcl}
&&\phi(a_1x^{l_1}_{i_1}b_1x^{l_2}_{i_2}b_2\cdots b_{m-1}x^{l_m}_{i_m}a_2)\p\\
&=&\mu_{y_1,...,y_{k'}} (X_{s_{1,1}+n_1}\cdots X_{s_{1,t_1}+n_1}  X^{l_1}_{i_1+n_1}\cdots X^{l_m}_{i_m+n_1} X_{s_{1,1}+n_1}\cdots X_{s_{2,t_2}+n_1})\pi(\p')  \\
&=&\p\mu_{y_1,...,y_{n'}}\otimes \pi (\alpha^{(m)}_{n',k'}(X_{s_{1,1}+n_1}\cdots X_{s_{1,t_1}+n_1}  X^{l_1}_{i_1+n_1}\cdots X^{l_m}_{i_m+n_1} X_{s_{1,1}+n_1}\cdots X_{s_{2,t_2}+n_1})) \p\\
&=&\sum\limits_{j_1,...,j_m=1}^n \mu_{y_1,...,y_{n'}}(X_{s_{1,1}+n_1}\cdots X_{s_{1,t_1}+n_1} X^{l_1}_{j_1+n_1}X_{r_{1,1}+n_1+n-k}\cdots\\ &&X_{r_{m-1,t'_{m-1}+n_1}+n-k}X^{l_m+n_1}_{j_m} X_{s_{1,1}+n_1}\cdots X_{s_{2,t_2}})
\p u_{j_1,i_1}\cdots u_{j_m,i_m}\p\\
\end{array}
$$
Notice that $(y_1,...,y_{n'})$ is spreadable and $ n+1\leq r_{,}$, the above equation becomes

$$
\begin{array}{crl}
&&\phi(a_1x^{l_1}_{i_1}b_1x^{l_2}_{i_2}b_2\cdots b_{m-1}x^{l_m}_{i_m}a_2)\p\\
&=&\sum\limits_{j_1,...,j_m=1}^n \mu_{y_1,...,y_{n'}}(X_{s_{1,1}+n_1}\cdots X_{s_{1,t_1}+n_1} X^{l_1}_{j_1+n_1}X_{r_{1,1}+n_1}\cdots\\ &&X_{r_{m-1,t'_{m-1}+n_1}}X^{l_m}_{j_m+n_1} X_{s_{1,1}+n_1}\cdots X_{s_{2,t_2}})
\p u_{j_1,i_1}\cdots u_{j_m,i_m}\p\\
&=&\sum\limits_{j_1,...,j_m=1}^n \phi(x_{s_{1,1}}\cdots x_{s_{1,t_1}}  x^{l_1}_{j_1}x_{r_{1,1}}\cdots x_{r_{m-1,t'_{m-1}}}x^{l_m}_{j_m} x_{s_{1,1}}\cdots x_{s_{2,t_2}})\\
&&\p u_{j_1,i_1}\cdots u_{j_m,i_m}\p\\
&=&\sum\limits_{j_1,...,j_m=1}^n\phi(a_1x^{l_1}_{j_1}b_1x^{l_2}_{j_2}b_2\cdots b_{m-1}x^{l_m}_{j_m}a_2)\p u_{j_1,i_1}\cdots u_{j_m,i_m}\p
\end{array}
$$
 
The proof is complete.
\end{proof}

\begin{lemma}\label{7.2}
Fix $n>k\in\mathbb{N}$, let $(u_{i,j})_{i=1,...,n;\, j=1,...,k}$ be the standard generators of $M_i(n,k)$. Then, we have 
$$E^+[x^{l_1}_{i_1}b_1x^{l_2}_{i_2}b_2\cdots b_{m-1}x^{l_m}_{i_m}]\otimes \p=\sum\limits_{j_1,...,j_m=1}^nE^+[x^{l_1}_{j_1}b_1x^{l_2}_{j_2}b_2\cdots b_{m-1}x^{l_m}_{j_m}]\otimes \p u_{j_1,i_1}\cdots u_{j_m,i_m}\p,$$
where $1\leq i_1,...i_m\leq k$,  $b_1,...,b_{m-1}\in A_{(n+1)+}$.
\end{lemma}
\begin{proof}
It is necessary to check the two sides of the equation equal to each other pointwisely, i.e. 
\begin{equation}\label{3}
\phi(a_1 E^+[x^{l_1}_{i_1}b_1x^{l_2}_{i_2}b_2\cdots b_{m-1}x^{l_m}_{i_m}]a_2)\p=\sum\limits_{j_1,...,j_m=1}^n \phi(a_1E^+[x^{l_1}_{j_1}b_1x^{l_2}_{j_2}b_2\cdots b_{m-1}x^{l_m}_{j_m}]a_2) \p u_{j_1,i_1}\cdots u_{j_m,i_m}\p
\end{equation}
for all $a_1,a_2\in A_{[-\infty,\infty]}$.
Given $a_1,a_2\in A_{[-\infty,\infty]}$, then there exists $M\in\mathbb{N}$ such that $a_1,a_2\in A_{M-}$. Then,   $$\alpha^{-m}(a_1),\alpha^{-m}(a_2)\in A_{0-}$$
for all $m>M$. By Lemma \ref{7.1}, we have 
$$
\begin{array}{rcl}
&&\phi(\alpha^{-m}(a_1)x^{l_1}_{i_1}b_1x^{l_2}_{i_2}b_2\cdots b_{m-1}x^{l_m}_{i_m}\alpha^{-m}(a_2))\p\\
&=&\sum\limits_{j_1,...,j_m=1}^n\phi(\alpha^{-m}(a_1)x^{l_1}_{j_1}b_1x^{l_2}_{j_2}b_2\cdots b_{m-1}x^{l_m}_{j_m}\alpha^{-m}(a_2))\p u_{j_1,i_1}\cdots u_{j_m,i_m}\p.
\end{array}
$$
Therefore, for all $m>M$,we have
$$
\begin{array}{rcl}

&&\phi(a_1\alpha^{m}(x^{l_1}_{i_1}b_1x^{l_2}_{i_2}b_2\cdots b_{m-1}x^{l_m}_{i_m})a_2)\p\\&=&\sum\limits_{j_1,...,j_m=1}^n\phi(a_1\alpha^m(x^{l_1}_{j_1}b_1x^{l_2}_{j_2}b_2\cdots b_{m-1}x^{l_m}_{j_m})a_2)\p u_{j_1,i_1}\cdots u_{j_m,i_m}\p.
\end{array}
$$
Let $m$ go to $+\infty$, we get equation  \ref{3}.\\
The proof is complete since $a_1,a_2$ are arbitrary.
\end{proof}

\begin{proposition}
Let $(\A,\phi)$ be a $W^*$-probability space, $(x_i)_{i\in\mathbb{Z}}$ a sequence of selfadjoint random variables in $\A$ , $E^+$ be the conditional expectation onto the positive tail algebra $\ATP$. Assume that the joint distribution of $(x_i)_{i\in\mathbb{Z}}$ is monotonically spreadable, then the same is true for the joint distribution with respect to $E^+$, i.e.  for fixed $n>k\in\mathbb{N}$ and $(u_{i,j})_{i=1,...,n;\, j=1,...,k}$ the standard generators of $M_i(n,k)$, we have that 
$$E^+[x^{l_1}_{i_1}b_1x^{l_2}_{i_2}b_2\cdots b_{m-1}x^{l_m}_{i_m}]\otimes \p=\sum\limits_{j_1,...,j_m=1}^nE^+[x^{l_1}_{i_1}b_1x^{l_2}_{i_2}b_2\cdots b_{m-1}x^{l_m}_{i_m}]\otimes \p u_{j_1,i_1}\cdots u_{j_m,i_m}\p,$$
 $1\leq i_1,...,i_m\leq k$, $l_1,...,l_m\in\mathbb{N} $ and $b_1,...,b_n\in\ATP$.
\end{proposition}
\begin{proof}
Since $b_1,...,b_{m-1}\in \ATP\in \A_n^+$, by Kaplansky's theorem,  there exists sequences $$\{b_{s,t}\}_{s=1,...m-1; t\in\mathbb{N}} \subset A_{n+}$$ such that $\|b_{s,t}\|\leq \|b_s\| $ and $\lim\limits_{n\rightarrow\infty} b_{s,t}=b_s$ in SOT for each $s=1,...,m-1$. Therefore, 
$$SOT-\lim\limits_{t_1\rightarrow\infty}x^{l_1}_{i_1}b_{1,t_1}x^{l_2}_{i_2}b_{2,t_2}\cdots b_{m-1,t_m}x^{l_m}_{i_m}= x^{l_1}_{i_1}b_{1}x^{l_2}_{i_2}b_{2,t_2}\cdots b_{m-1,t_m}x^{l_m}_{i_m}.$$
By Lemma \ref{7.2}, we have 
$$
E^+[x^{l_1}_{i_1}b_{1,t_1}x^{l_2}_{i_2}b_{2,t_2}\cdots b_{m-1,t_m}x^{l_m}_{i_m}]\otimes \p=\sum\limits_{j_1,...,j_m=1}^nE^+[x^{l_1}_{j_1}b_{1,t_1}x^{l_2}_{j_2}b_{2,t_2}\cdots b_{m-1,t_{m-1}}x^{l_m}_{j_m}]\otimes \p u_{j_1,i_1}\cdots u_{j_m,i_m}\p$$
Let $t_1$ go to $+\infty$, by normality of $E^+$, we have 
$$
E^+[x^{l_1}_{i_1}b_{1}x^{l_2}_{i_2}b_{2,t_2}\cdots b_{m-1,t_m}x^{l_m}_{i_m}]\otimes \p=\sum\limits_{j_1,...,j_m=1}^nE^+[x^{l_1}_{j_1}b_{1}x^{l_2}_{j_2}b_{2,t_2}\cdots b_{m-1,t_{m-1}}x^{l_m}_{j_m}]\otimes \p u_{j_1,i_1}\cdots u_{j_m,i_m}\p$$
Again, take $t_2,...,t_{m-1}$ to $+\infty$, we have 
\begin{equation}\label{5}
E^+[x^{l_1}_{i_1}b_1x^{l_2}_{i_2}b_2\cdots b_{m-1}x^{l_m}_{i_m}]\otimes \p=\sum\limits_{j_1,...,j_m=1}^nE^+[x_{j_1}b_1x_{j_2}b_2\cdots b_{m-1}x_{j_m}]\otimes \p u_{j_1,i_1}\cdots u_{j_m,i_m}\p
\end{equation}
\end{proof}
According to the universal conditions of $M_i(n,k)$, if $i_s=i_{s+1}$ for some $s$, then the terms on the right hand side are not vanished only if $j_s=j_{s+1}$. Therefore we can shorten the product on the right hand side of \ref{5} if $i_s=i_{s+1}$ for some $s$. We have

\begin{proposition}\label{spreadable polynomial}
Let $(\A,\phi)$ be a $W^*$-probability space, $(x_i)_{i\in\mathbb{Z}}$ a sequence of selfadjoint random variables in $\A$ , $E^+$ be the conditional expectation onto the positive tail algebra $\ATP$. Assume that the joint distribution of $(x_i)_{i\in\mathbb{Z}}$ is monotonically spreadable,  for fixed $n>k\in\mathbb{N}$ and $(u_{i,j})_{i=1,...,n;\, j=1,...,k}$ the standard generators of $M_i(n,k)$, we have that 
$$E^+[p_1(x_{i_1})\cdots p_m(x_{i_m})]\otimes \p=\sum\limits_{j_1,...,j_m=1}^nE^+[p_1(x_{j_1})\cdots p_m(x_{j_m})]\otimes \p u_{j_1,i_1}\cdots u_{j_m,i_m}\p,$$
whenever $1\leq i_1,...,i_m\leq k$, $i_1\neq\cdots \neq i_m$ and $p_1,...,p_m\in\ATP\langle X\rangle_0$.
\end{proposition}

\begin{lemma}\label{maximal index}
Let $(\A,\phi)$ be a $W^*$-probability space, $(x_i)_{i\in\mathbb{Z}}$ a sequence of selfadjoint random variables in $\A$ , $E^+$ be the conditional expectation onto the positive tail algebra $\ATP$. Assume that the joint distribution of $(x_i)_{i\in\mathbb{Z}}$ is monotonically spreadable, then 
$$E^+[p_1(x_{i_1})\cdots p_s(x_{i_s})\cdots  p_m(x_{i_m})] =E^+[p_1(x_{i_1})\cdots E^+[p_s(x_{i_s})]\cdots  p_m(x_{i_m})] $$
whenever $i_s>i_t$ for all $t\neq s$, $i_1\neq\cdots \neq i_m$ and $p_1,...,p_m\in\ATP\langle X\rangle_0 $.
\end{lemma}
\begin{proof}
Since $(x_i)_{i\in\mathbb{Z}}$ is spreadable, by Lemma \ref{fixed point},  we have that 
$$\alpha (p_t(x_{i_t}))=p_t(\alpha(x_{i_t}))$$
and 
$$E^+[\alpha^{k'}(a)]=E^+[a]$$ 
for all $a\in\bigcup\limits_{n'\in\mathbb{Z}}\A^+_{n'}$ and $k'\in \mathbb{Z}$.\\
Therefore, it is sufficient to prove the statement under the assumption that $i_1,...,i_m>0$. Let $i_s=k$, $(u_{i,j})_{i=1,...,n+1;\, j=1,...,k}$ the standard generators of $M_i(n+k,k)$. By proposition \ref{spreadable polynomial}, we have 
$$E^+[p_1(x_{i_1})\cdots p_m(x_{i_m})]\otimes \p=\sum\limits_{j_1,...,j_m=1}^{n+k}E^+[p_1(x_{j_1})\cdots p_m(x_{j_m})]\otimes \p u_{j_1,i_1}\cdots u_{j_m,i_m}\p.$$
Now, apply proposition \ref{nontrivial representation} by letting $l_1=\cdots=l_{k-1}=1$ and $l_k=n+1$, then we have 
$$E^+[p_1(x_{i_1})\cdots p_s(x_{i_s})\cdots  p_m(x_{i_m})] \otimes \p=\frac{1}{n+1}\sum\limits_{j_s=k}^{n+k} E^+[p_1(x_{i_1})\cdots p_s(x_{j_s})\cdots  p_m(x_{i_m})] \otimes\p.$$
Since $n$ is arbitrary, and $E^+$ is normal on $\A_0^+$, we have 
$$
\begin{array}{rcl}
&&E^+[p_1(x_{i_1})\cdots p_s(x_{i_s})\cdots  p_m(x_{i_m})] \\
&=&\frac{1}{n+1}\sum\limits_{j_s=k}^{n+k} E^+[p_1(x_{i_1})\cdots p_s(x_{j_s})\cdots  p_m(x_{i_m})]\\
&=&\text{WOT}-\lim\limits_{n\rightarrow \infty} E^+[p_1(x_{i_1})\cdots (\frac{1}{n+1}\sum\limits_{j_s=k}^{n+k}p_s(x_{j_s}))\cdots  p_m(x_{i_m})]\\
&=&\text{WOT}-\lim\limits_{n\rightarrow \infty} E^+[p_1(x_{i_1})\cdots (\frac{1}{n+1}\sum\limits_{t=0}^{n}\alpha^t(p_s(x_{i_s}))\cdots  p_m(x_{i_m})]\\
&=&\text{WOT}-\lim\limits_{n\rightarrow \infty} E^+[p_1(x_{i_1})\cdots E^+[p_s(x_{i_s})]\cdots  p_m(x_{i_m})].
\end{array}
$$
The proof is complete.
\end{proof}

Now, we turn to consider the case that  the maximal index is not unique.
\begin{proposition}\label{ms is mi}
Let $(\A,\phi)$ be a $W^*$-probability space, $(x_i)_{i\in\mathbb{Z}}$ a sequence of selfadjoint random variables in $\A$ , $E^+$ be the conditional expectation onto the positive tail algebra $\ATP$. Assume that the joint distribution of $(x_i)_{i\in\mathbb{Z}}$ is monotonically spreadable, then 
$$E^+[p_1(x_{i_1})\cdots p_s(x_{i_s})\cdots  p_m(x_{i_m})] =E^+[p_1(x_{i_1})\cdots E^+[p_s(x_{i_s})]\cdots  p_m(x_{i_m})] $$
whenever $i_s=\max\{i_1,...,i_n\}$ for all $t\neq s$, $i_1\neq\cdots \neq i_m$ and $p_1,...,p_m\in\ATP\langle X\rangle_0 $.
\end{proposition}
\begin{proof}
Again, we can assume that $i_1,...,i_t>0$ and $\max\{i_1,...,i_m\}=k$.  Suppose the number $k$ appears $t$ times in the sequence, which are $\{i_{l                                                                                                                                          _j}\}_j=1,...,t$ such that $i_{l_j}=k$ and $l_1< l_2<\cdots<l_t$.  Fix $n,k$ and consider $M_i(n+k,k)$, by proposition \ref{spreadable polynomial} and proposition \ref{nontrivial representation}, we have 

$$
\begin{array}{rcl}
& &E^+[p_1(x_{i_1})\cdots p_{l_1}(x_{i_{l_1}})\cdots p_{l_2}(x_{i_{l_2}})\cdots  p_m(x_{i_m})]\otimes P\\
&=&\sum\limits_{j_{l_1},j_{l_2},...j_{l_t}=k}^{k+n} E^+[p_1(x_{i_1})\cdots p_{l_1}(x_{j_{l_1}})\cdots p_{l_2}(x_{j_{l_2}})\cdots  p_m(x_{i_m})]\otimes PP_{j_{l_1},k}PP_{j_{l_2},k}P\cdots u_{j_{l_t},k}P \\
&=&\frac{1}{(n+1)^t}\sum\limits_{j_{l_1},j_{l_2},...j_{l_t}=k}^{k+n} E^+[p_1(x_{i_1})\cdots p_{l_1}(x_{j_{l_1}})\cdots p_{l_2}(x_{j_{l_2}})\cdots  p_m(x_{i_m})]]\otimes P \\
&=&\frac{1}{(n+1)^t}(\sum\limits_{j_{l_s}\neq j_{l_r}\, \text{if}\, s\neq r}^{N} E^+[p_1(x_{i_1})\cdots p_{l_1}(x_{j_{l_1}})\cdots p_{l_2}(x_{j_{l_2}})\cdots  p_m(x_{i_m})]\otimes P\\
&& + \sum\limits_{j_{l_s}=j_{l_t}\, \text{for some}\, s\neq t}^{N} E^+[p_1(x_{i_1})\cdots p_{l_1}(x_{j_{l_1}})\cdots p_{l_2}(x_{j_{l_2}})\cdots  p_m(x_{i_m})]\otimes P)\\
\end{array}
$$
In the first part of the sum, apply proposition \ref{maximal index} on indices $j_{l_1},...j_{l_{l_t}}$ recursively, it follows that
$$
E^+[p_1(x_{i_1})\cdots p_s(x_{j_{l_1}})\cdots p_s(x_{j_{l_2}})\cdots  p_m(x_{i_m})]=E^+[p_1(x_{i_1})\cdots E[p_{l_1}(x_{j_{l_1}})]\cdots E[p_{l_2}(x_{j_{l_2}})]\cdots  p_m(x_{i_m})].
$$
Since   $E[p_s(x_{j_{l_1}})]= E[p_s(x_{k})]$, for all $j_{l_1},...,j_{l_t}$,
$$
E^+[p_1(x_{i_1})\cdots p_{l_1}(x_{j_{l_1}})\cdots p_{l_2}(x_{j_{l_2}})\cdots  p_m(x_{i_m})]=E^+[p_1(x_{i_1})\cdots E[p_{l_1}(x_{k})]\cdots E[p_{l_2}(x_{k})]\cdots  p_m(x_{i_m})].
$$
Then, we have 
$$\begin{array}{rcl}
&&\frac{1}{(n+1)^t}(\sum\limits_{j_{l_s}\neq j_{l_r}\, \text{if}\, s\neq r}^{N} E^+[p_1(x_{i_1})\cdots p_{l_1}(x_{j_{l_1}})\cdots p_{l_2}(x_{j_{l_2}})\cdots  p_m(x_{i_m})]\otimes P\\
&=&
\frac{\prod\limits_{s=0}^{t-1}(n+1-s)}{{n+1}^t}E^+[p_1(x_{i_1})\cdots E[p_{l_1}(x_{k})]\cdots E[p_{l_2}(x_{k})]\cdots  p_m(x_{i_m})]\otimes P\\
\end{array},
$$
which converges to $E[p_s(x_{k})]\cdots E[p_s(x_{k})]\cdots  p_m(x_{i_m})]\otimes P$ in norm as $n$ goes to $+\infty$.\\

To the second part of the sum, we have 
$$\begin{array}{rcl}
&&\|E^+[p_1(x_{i_1})\cdots p_{l_1}(x_{j_{l_1}})\cdots p_{l_2}(x_{j_{l_2}})\cdots  p_m(x_{i_m})]\|\\
&\leq&\|p_1(x_{i_1})\cdots p_{l_1}(x_{j_{l_1}})\cdots p_{l_2}(x_{j_{l_2}})\cdots  p_m(x_{i_m})\|\\
&\leq&\|p_1(x_{i_1})\|\cdots \|p_{l_1}(x_{j_{l_1}})\|\cdots \|p_{l_2}(x_{j_{l_2}})\|\cdots  \|p_m(x_{i_m})\|\\
&\leq&\|p_1(x_{1})\|\cdots \|p_{l_1}(x_{1})\|\cdots \| p_{l_2}(x_{1})\|\cdots  \| p_m(x_{1})\|\\
\end{array}$$
which is finite. Therefore,
$$
\begin{array}{rcl}
&&|\sum\limits_{j_{l_s}=j_{l_t}\, \text{for some}\, s\neq t}^{N} E^+[p_1(x_{i_1})\cdots p_{l_1}(x_{j_{l_1}})\cdots p_{l_2}(x_{j_{l_2}})\cdots  p_m(x_{i_m})]\|  \\  &\leq& (1-\frac{\prod\limits_{s=0}^{t-1}(n+1-s)}{(n+1)^t})\|p_1(x_{1})\|\cdots \|p_{l_1}(x_{1})\|\cdots \| p_{l_2}(x_{1})\|\cdots  \| p_m(x_{1})\|
\end{array}
$$
goes to $0$ as $n$ goes to $+\infty$.

Therefore, we have 
$$E^+[p_1(x_{i_1})\cdots p_{l_1}(x_{i_{l_1}})\cdots p_{l_2}(x_{i_{l_2}})\cdots  p_m(x_{i_m})]=E^+[p_1(x_{i_1})\cdots E[p_{l_1}(x_{k})]\cdots E[p_{l_2}(x_{k})]\cdots  p_m(x_{i_m})]$$
The same we can show that
$$
\begin{array}{rcl}
&&E^+[p_1(x_{i_1})\cdots p_{l_1}(x_{k})\cdots E^+[ p_{s}(x_{i_s})]\cdots p_{l_2}(x_{k})\cdots  p_m(x_{i_m})]\\
&=&E^+[p_1(x_{i_1})\cdots E[p_{l_1}(x_{k})]\cdots E[p_{l_2}(x_{k})]\cdots  p_m(x_{i_m})]\\
\end{array}
$$
which implies

$$E^+[p_1(x_{i_1})\cdots p_s(x_{i_s})\cdots  p_m(x_{i_m})] =E^+[p_1(x_{i_1})\cdots E^+[p_s(x_{i_s})]\cdots  p_m(x_{i_m})] $$
\end{proof}

\section{de Finetti type theorem for monotone spreadability }
\subsection{Proof of main theorem 1}
Now, we turn to prove our main theorem for monotone independence:
\begin{theorem}
Let $(\A,\phi)$ be a non degenerated $W^*$-probability space and $(x_i)_{i\in\mathbb{Z}}$ be a bilateral infinite sequence of selfadjoint random variables which generate $\A$. Let $\A^+_k$ be the WOT closure of the non-unital algebra generated by $\{x_i|i\geq k\}$. Then the following are equivalent:
\begin{itemize}
\item[a)] The joint distribution of $(x_i)_{i\in \mathbb{Z}}$  is monotonically spreadable. 

\item[b)]  For all $k\in\mathbb{Z}$, there exits a $\phi$ preserving conditional expectation $E_k:\A^+_k\rightarrow \ATP$ such that the sequence  $(x_i)_{i\geq k}$ is identically distributed and monotone  with respect $E_k$.  Moreover, $E_k|_{\A_{k'}}=E_{k'}$ when $k\geq k'$.
\end{itemize}
\end{theorem}
\begin{proof}
\lq\lq $b)\Rightarrow a) $ \rq\rq follows corollary \ref{mi is ms}\\

We will prove \lq\lq $a)\Rightarrow b) $ \rq\rq by induction.  Since the sequence is spreadable, it is suffices to prove $a)\Rightarrow b) $ for $k=1$:\\
By the results in the previous two sections, there exists a conditional expectation $E_k:\A^+_k\rightarrow \ATP$ such that the sequence  $(x_i)_{i\geq k}$ is identically distributed with respect to $E_k$ and $E_k|_{\A_{k'}}=E_{k'}$ when $k\geq k'$.  
Actually, $E_k$ is the restriction of $E^+$ on $\A^+_k$. 
Since the sequence is spreadable, we just need to show that the sequence  $(x_i)_{i\in\mathbb{N}}$ is  monotonically independent with respect to $E_1$, i.e.
\begin{equation}\label{equaution 6}
E^+[p_1(x_{i_1})\cdots p_s(x_{i_s})\cdots  p_m(x_{i_m})] =E^+[p_1(x_{i_1})\cdots E^+[p_s(x_{i_s})]\cdots  p_m(x_{i_m})]
\end{equation} 
 $i_{s-1}<i_s>i_{s+1}$, $i_1\neq\cdots \neq i_m$, $i_1,...,i_m\in\mathbb{N}$ and $p_1,...,p_m\in\ATP\langle X\rangle $.\\
 Now, we prove this equality by induction on the maximal index of $\{i_1,...,i_m\}$:\\
 When $\max\{i_1,...,i_m\}=1$, then  equality is true because $i_s=1$ and the length of the sequence $(i_1,...,i_m)$ can only be $1$.\\
 Suppose the equality holds for $\max\{i_1,...,i_m\}=n$. When $\max\{i_1,...,i_m\}=n+1$, we have two cases:\\
{\bf Case 1: $i_s=n+1$.} In this case  the equality follows proposition \ref{ms is mi}.\\

\noindent{\bf Case 2: $i_s\leq n$.} Suppose the number $n+1$ appears $t$ times in the sequence, which are $\{i_{l                                                                                                                                          _j}\}_j=1,...,t$ such that $i_{l_j}=k$ and $l_1< l_2<\cdots<l_t$.  Since $i_{s-1}<i_s>i_{s+1}$, $i_{s-1},i_s,i_{s+1}\neq n+1$. By proposition \ref{ms is mi}, we have:
$$
\begin{array}{rcl}
&& E^+[p_1(x_{i_1})\cdots p_{l_1}(x_{i_{l_1}}) \cdots p_{s-1}(x_{i_{s-1}}) p_s(x_{i_s})p_{s+1}(x_{i_{s+1}}) \cdots p_{l_t}(x_{i_{l_t}})\cdots  p_m(x_{i_m})]\\ 
&=&E^+[p_1(x_{i_1})\cdots E^+[p_{l_1}(x_{i_{l_1}})] \cdots p_{s-1}(x_{i_{s-1}}) p_s(x_{i_s})p_{s+1}(x_{i_{s+1}}) \cdots E^+[p_{l_t}(x_{i_{l_t}})]\cdots  p_m(x_{i_m}))]\\ 
\end{array}
$$
Notice that 
$$p_1(x_{i_1})\cdots E^+[p_{l_1}(x_{i_{l_1}})] \cdots p_{s-1}(x_{i_{s-1}}) p_s(x_{i_s})p_{s+1}(x_{i_{s+1}}) \cdots E^+[p_{l_t}(x_{i_{l_t}})]\cdots  p_m(x_{i_m})\in\ATP\langle X_1,...,X_n\rangle$$
by induction, we have 
$$
\begin{array}{rcl}
&&E^+[p_1(x_{i_1})\cdots E^+[p_{l_1}(x_{i_{l_1}})] \cdots p_{s-1}(x_{i_{s-1}}) p_s(x_{i_s})p_{s+1}(x_{i_{s+1}}) \cdots E^+[p_{l_t}(x_{i_{l_t}})\cdots  p_m(x_{i_m})]\\ 
&=&E^+[p_1(x_{i_1})\cdots E^+[p_{l_1}(x_{i_{l_1}})] \cdots p_{s-1}(x_{i_{s-1}}) E^+[p_s(x_{i_s})]p_{s+1}(x_{i_{s+1}}) \cdots E^+[p_{l_t}(x_{i_{l_t}})]\cdots  p_m(x_{i_m})]\\ 
&=& E^+[p_1(x_{i_1})\cdots p_{l_1}(x_{i_{l_1}}) \cdots p_{s-1}(x_{i_{s-1}}) E^+[p_s(x_{i_s})]p_{s+1}(x_{i_{s+1}}) \cdots p_{l_t}(x_{i_{l_t}})\cdots  p_m(x_{i_m})]
\end{array}
$$
The last equality follows proposition \ref{ms is mi}. This our desired conclusion.

\end{proof}

\subsection{ Conditional expectation $E^-$}  
We do not know whether we can extend $E^+$ to the whole space $\A$. But,  the conditional expectation $E^-$ can be extended to the whole algebra $\A$ if the bilateral sequence $(x_i)_{i\in\mathbb{Z}}$ is monotonically spreadable.   Given $a,b,c\in A_{[-\infty,\infty]}$, then there exists $L\in\mathbb{N}$ such that $a,b,c\in A_{[-L,L]}$. Therefore, $\alpha^{-3L}(c)\in \A_{[-4L,-3L]}$.   Since $(x_{-4L},x_{-4L+1},...)$ is monotonically with respect to $E^+$,  we have 
$$ \begin{array}{rcl}
&&\phi(aE^-[b] c)\\
&=&\lim\limits_{n\rightarrow \infty}\phi(a\alpha^{-n}(b)c)\\
&=&\lim\limits_{n\rightarrow \infty,n>4L}\phi(a\alpha^{-n}(b)c)\\
&=&\lim\limits_{n\rightarrow \infty,n>4L}\phi(E^+[a\alpha^{-n}(b)c])\\
&=&\lim\limits_{n\rightarrow \infty,n>4L}\phi(E^+[E^+[a]\alpha^{-n}(b) E^+[c]])\\
&=&\lim\limits_{n\rightarrow \infty}\phi(E^+[a]\alpha^{-n}(b) E^+[c])\\
&=&\lim\limits_{n\rightarrow \infty}\phi(E^+[a]E^-[b]  E^+[c])\\
\end{array}
$$

Since $\A$ is generated by  countablely many operators, by Kaplansky's density theorem,  for all $y\in\A$, there exists a sequence $\{y_n\}_{n\in\mathbb{N}}\subset A_{[-\infty,\infty]}$ such that $\|y_n\|\leq \|y\|$ for all $n$ and $y_n$ converges to $y$ in WOT. Then,  for all $a,c\in  A_{[-\infty,\infty]}$ we have 
$$
\lim\limits_{n\rightarrow \infty}\phi(aE^-[y_n]c)=\lim\limits_{n\rightarrow \infty}\phi(E^+[a]y_nE^+[c])=\phi(E^+[a]yE^+[c])
$$
Therefore, $ E^-[y_n]$ converges to an element $y'$ pointwisely. Moreover, $y'$ depends only on $y$.  If we define $E^-[y]=y'$, then we have 
\begin{proposition}\label{extention1}
Let $(\A,\phi)$ be a non-degenerated $W^*$-probability space and $(x_i)_{i\in\mathbb{Z}}$ be a bilateral infinite sequence of selfadjoint random variables which generate $\A$. If $(x_i)_{i\in\mathbb{Z}}$ is monotonically spreadable, then the negative conditional expectation $E^-$ can be extend to the whole algebra $\A$ such that 
$$\phi(aE^-[y]b)=\phi(E^+[a]yE^+[c])$$
for all $y\in\A$ and $a,c\in A_{[-\infty,\infty]}$. Moreover, the extension is normal.
\end{proposition}

\section{de Finetti type theorem for boolean spreadability}
In this section, we assume that $(\A, \phi)$ is a $W^*$-probability space with a non-degenerated normal state and $\A$ is generated by a bilateral sequence of random variables $(x_i)_{i\in\mathbb{Z}}$ and $(x_i)_{i\in\mathbb{Z}}$ are boolean spreadable. 
\begin{lemma}\label{reverse}  Let $y_i=x_{-i}$ for all $i\in\mathbb{Z}$, then $(y_i)_{i\in\mathbb{Z}}$ is also boolean spreadable.
\end{lemma}
\begin{proof}
By proposition \ref{sub spreadable}, it suffices to show that $(y_{i})_{i=1,...,n}$ is boolean spreadable for all $n\in\mathbb{N}$.  Given a  natural number $k<n$, assume the standard generators of $B_i(n,k)$ are $\{u_{i,j}\}_{i=1,...,n;j=1,...,k}$ and invariant projection $\p$. \\
Consider the matrix $\{u'_{i,j}\}_{i=1,...,n;j=1,...,k}$ such that $u'_{i,j}=u_{n+1-i,k+1-j}.$.  it is obvious that the entries of the matrix are are orthogonal projections and 
$$\sum\limits_{i=1}^nu'_{i,j}\p=  \sum\limits_{i=1}^nu_{i,k+1-j}\p=\p.$$
Given $j,j',i,i'\in \mathbb{N}$ such that $ 1\leq j <j'\leq k$ and $1\leq i\leq i'\leq n$. Then, we have $n+1-i\leq n+1-i'$ and $k+1-j<k+1-j'$. Therefore,
$$ u'_{i,j}u'_{i',j'}=u_{n+1-i,k+1-j}u_{n+1-i',k+1-j'}=0.$$
It implies that $\{u'_{i,j}\}_{i=1,...,n;j=1,...,k}$ and $\p$ satisfy the universal conditions of $B_i(n,k)$. It follows that there exists a unital $C^*$-homomorphism $\Phi:B_i(n,k)\rightarrow B_i(n,k) $ such that:
$$\Phi(u_{i,i})=u'_{i,j},\,\,\,\text{and}\,\,\Phi(\p)=\p.$$
Let $z_i=x_{i-n-1}$ for $i=1,...,n$. Since $(x_i)_{i\in\mathbb{Z}}$ are boolean spreadable, $(z_i)_{i=1,...,n}$ is boolean spreadable. Therefore, for $i_1,...,i_L\in [k]$, we have
$$
\begin{array}{rcl}
&&\phi(y_{i_1}\cdots y_{i_L})\p\\
&=& \phi(y_{n-k+i_1}\cdots y_{n-k+i_L})\p\\
&=&\phi(x_{-n+k-i_1}\cdots x_{n-k-i_L})\p\\
&=&\Phi( \phi(z_{k+1-i_1}\cdots z_{k+1-i_L})\p)\\
&=&\Phi(\sum\limits_{j_1,...,j_L=1}^n\phi(z_{j_1}\cdots z_{j_L})\p u_{j_1, k+1-i_1}\cdots u_{j_L,k+1-i_L}\p)\\
&=&\sum\limits_{j_1,...,j_L=1}^n\phi(z_{j_1}\cdots z_{j_L})\p u_{n+1-j_1,i_1}\cdots u_{n+1-j_L,i_L}\p\\
&=&\sum\limits_{j_1,...,j_L=1}^n\phi(x_{j_1-n-1}\cdots x_{j_L-n-1})\p u_{n+1-j_1,i_1}\cdots u_{n+1-j_L,i_L}\p\\
&=&\sum\limits_{j_1,...,j_L=1}^n\phi(y_{n+1-j_1}\cdots y_{n+1-j_L})\p u_{n+1-j_1,i_1}\cdots u_{n+1-j_L,i_L}\p\\
&=&\sum\limits_{j_1,...,j_L=1}^n\phi(y_{j_1}\cdots y_{j_L})\p u_{j_1,i_1}\cdots u_{j_L,i_L}\p\\
\end{array}
$$
which completes the proof.
\end{proof}

\begin{proposition}\label{equal tail algebra}

$(\A, \phi)$ is a $W^*$-probability space with a non-degenerated normal state and $\A$ is generated by a bilateral sequence of random variables $(x_i)_{i\in\mathbb{Z}}$ and $(x_i)_{i\in\mathbb{Z}}$ are boolean spreadable. Then, $E^-$ and $E^+$ can be extend to the whole algebra $\A$. Moreover, $E^-=E^+$
\end{proposition}
\begin{proof}
Since $(x_i)_{i\in\mathbb{Z}}$ is boolean spreadable, $(x_i)_{i\in\mathbb{Z}}$ is monotonically spreadable. By proposition \ref{extention1} $E^-$ can be extended to the whole algebra. By Lemma \ref{reverse}, $(x_{-i})_{i\in\mathbb{Z}}$ is also boolean  spreadable and its negative-conditional expectation is exactly the positive conditional expectation of $(x_{i})_{i\in\mathbb{Z}}$. Therefore, $E^+$ can also be extended the whole algebra $\A$ normally. Give $a,b,c\in A_{[-\infty,\infty]}$, by Lemma \ref{extention1}, we have 
%$$\phi(aE^-[b] c)=\phi(E^+[a]b E^+[c])=\phi(E^+[E^+[a]b E^+[c]])=\phi(E^+[a]E^+[b] E^+[c])$$
 
 $$ \begin{array}{rcl}
\phi(aE^-[b] c)&=&\phi(E^+[a]b E^+[c])\\
&=&\phi(E^+[E^+[a]b E^+[c]])\\
&=&\phi(E^+[a]E^+[b] E^+[c])\\
&=&\lim\limits_{n\rightarrow \infty}\phi(\alpha^n(a)E^+[b] E^+[c])\\
&=&\lim\limits_{n\rightarrow \infty}\lim\limits_{m\rightarrow \infty}\phi(\alpha^n(a)E^+[b] \alpha^m(c))\\
\end{array}
$$
Notice that, for fixed $n,m$, 
$$\phi(\alpha^n(a)E^+[b] \alpha^m(c))=\phi(\alpha^n(a)\alpha^L(b) \alpha^m(c))
$$
for $L\in\mathbb{N}$ which is large enough.  Since $(x_{-i})_{i\in\mathbb{Z}}$ is monotonically spreadable,  by theorem \ref{theorem 1}, $(x_{-i})_{i\in\mathbb{Z}}$ is monotonically independent with respect to $E^-$. Therefore, we have
$$
\begin{array}{rcl}
&&\phi(\alpha^n(a)E^+[b] \alpha^m(c))\\
&=&\phi(\alpha^n(a)\alpha^L(b) \alpha^m(c))\\
&=&\phi(E^-[\alpha^n(a)\alpha^L(b) \alpha^m(c)])\\
&=&\phi(E^-[\alpha^n(a)]E^-[\alpha^L(b)] E^-[\alpha^m(c)])\\
&=&\phi(E^-[a]E^-[b] E^-[c])\\
&=&\phi(E^-[E^-[a]b E^-[c]])\\
&=&\phi(E^-[a]b E^-[c])\\
&=&\phi(aE^+[b]c)\\
\end{array}
$$

$$ \begin{array}{rcl}
\phi(aE^-[b] c)&=&\phi(E^+[a]b E^+[c])\\
&=&\lim\limits_{n\rightarrow \infty}\lim\limits_{m\rightarrow \infty}\phi(\alpha^n(a)E^+[b] \alpha^m(c))\\
&=&\lim\limits_{n\rightarrow \infty}\lim\limits_{m\rightarrow \infty}\phi(aE^+[b]c)\\
&=&\phi(aE^+[b]c)\\
\end{array}
$$
It implies that $E^+[b]=E^-[b]$ for all $b\in A_{[-\infty,\infty]}$. Since $\A$ is the WOT closure of $A_{[-\infty,\infty]}$, the proof is complete.
\end{proof}

\begin{corollary}
$(\A, \phi)$ is a $W^*$-probability space with a non-degenerated normal state and $\A$ is generated by a bilateral sequence of random variables $(x_i)_{i\in\mathbb{Z}}$ and $(x_i)_{i\in\mathbb{Z}}$ are boolean spreadable.  Then, the positive tail algebra and the negative tail algebra of $(x_i)_{i\in\mathbb{Z}}$ are the same.
\end{corollary}

Now, we are ready to prove theorem \ref{theorem 2}

\begin{theorem}
Let $(\A,\phi)$ be a non degenerated $W^*$-probability space and $(x_i)_{i\in\mathbb{Z}}$ be a bilateral infinite sequence of selfadjoint random variables which generate $\A$ as a von Neumann algebra. Then the following are equivalent:
\begin{itemize}
\item[a)] The joint distribution of $(x_i)_{i\in \mathbb{N}}$  is boolean spreadable. 

\item[b)]  The sequence $(x_i)_{i\in\mathbb{Z}}$ is identically distributed and boolean independent with respect to the $\phi-$preserving conditional expectation $E^+$ onto the non unital positive tail algebra of the $(x_i)_{i\in\mathbb{Z}}$
\end{itemize}
\end{theorem}
\begin{proof}
\lq\lq $b)\Rightarrow a)$\rq\rq. If the sequence $(x_i)_{i\in\mathbb{Z}}$ is identically distributed and boolean independent with respect to a $\phi-$preserving conditional expectation $E$ , then sequence $(x_i)_{i\in\mathbb{Z}}$ is boolean exchangeable by theorem 7.1  in \cite{Liu}.  According the diagram in section 4, $(x_i)_{i\in\mathbb{Z}}$ is boolean spreadable.\\

\lq\lq $a)\Rightarrow b)$\rq\rq. By Lemma\ref{equal tail algebra}, $(x_i)_{i\in\mathbb{Z}}$ is  monotone  with respect to $E^+$, $(x_{-i})_{i\in\mathbb{Z}}$ is  monotone  with respect to $E^-$ and $E^+=E^-$. Therefore, 
$$
\begin{array}{rcl}
&&E^+[p_1(x_{i_1})\cdots  p_m(x_{i_m})] =E^+[p_1(x_{i_1})]E^+[p_2(x_{i_2})\cdots  p_m(x_{i_m})]=\cdots\\
&=&E^+[p_1(x_{i_1})]E^+[p_2(x_{i_2})]\cdots\cdots  E^+[p_m(x_{i_m})]
\end{array}
$$
whenever  $i_1\neq\cdots \neq i_m$ and $p_1,...,p_m\in\ATP\langle X\rangle $. 
The proof is complete.

\end{proof}

\printbibliography
\noindent Department of Mathematics\\
University of California at Berkeley\\
Berkeley, CA 94720, USA\\
E-MAIL: weihualiu@math.berkeley.edu\\

\end{document}